\documentclass[12pt]{amsart}
\usepackage{amscd}      
\usepackage{amssymb}
\usepackage{rotating} 
 \usepackage{tikz}
\usetikzlibrary{patterns,snakes,arrows}

\usepackage{lscape}
\usepackage{latexsym, amsmath, amsthm, graphics, amsxtra, pb-diagram}

\usepackage{xypic}      
\LaTeXdiagrams          
\usepackage[all]{xy}
\xyoption{2cell} \UseAllTwocells \xyoption{frame} \CompileMatrices
\allowdisplaybreaks[3]

\usepackage{latexsym}
\usepackage{epsfig}
\usepackage{amsfonts}
\usepackage{enumerate}
\usepackage{times}

\theoremstyle{theorem}
\newtheorem{thm}{Theorem}[section]
\newtheorem{cor}[thm]{Corollary}
\newtheorem{lem}[thm]{Lemma}

\newtheorem{prop}[thm]{Proposition}

\theoremstyle{definition}
\newtheorem{defn}[thm]{Definition}
\newtheorem{rmk}[thm]{Remark}

\newtheorem{example}[thm]{Example}

\newtheorem*{thm*}{Theorem}

\numberwithin{equation}{section}

\theoremstyle{definition}

\theoremstyle{remark}

\theoremstyle{remark}

\numberwithin{equation}{section}

\newcommand{\M}{\mathcal{M}}
\newcommand{\Mbar}{\overline{\M}}

\def\<{\left\langle}
\def\>{\right\rangle}

\newcommand{\E}{\mathcal{E}}

\newcommand{\X}{\mathcal{X}}

\newcommand{\gG}{\mathcal{G}}

\newcommand{\twoarrows}{\rightrightarrows}

\newcommand\intF{\mathaccent23{F}}
\newcommand \Sigfan{\mathbf{\Sigma}}
\newcommand\XSigma{\X(\mathbf{\Sigma})}

\newcommand{\com}{\mathbb{C}}
\newcommand{\ration}{\mathbb{Q}}
\newcommand{\real}{\mathbb{R}}
\newcommand{\inte}{\mathbb{Z}}

\def\darr#1{\raise1.5ex\hbox{$\leftrightarrow$}
\mkern-16.5mu #1}

\def\roughly#1{\raise.3ex\hbox{$#1$\kern-.75em
\lower1ex\hbox{$\sim$}}}

\def\opname#1{\mathop{\kern0pt{\rm #1}}\nolimits}

\begin{document}
\title[Seidel representations and quantum cohomology of toric orbifolds]{Seidel representations\\ and quantum cohomology of toric orbifolds}
\author{Hsian-Hua Tseng}
\address{Department of Mathematics\\ Ohio State University\\ 100 Math Tower\\ 231 West 18th Avenue\\ Columbus, OH 43210-1174\\ USA}
\email{hhtseng@math.ohio-state.edu}

\author{Dongning Wang}
\address{Department of Mathematics \\ University of Wisconsin-Madison\\ Van Vleck Hall\\ 480 Lincoln Dr.\\ Madison, WI 53706\\ USA}
\email{dwang@math.wisc.edu}

\date{\today}

\begin{abstract}
We use Seidel representation for symplectic orbifolds constructed in \cite{TW} to compute the quantum cohomology ring of a compact symplectic toric orbifold $(\X,\omega)$. 
\end{abstract}
\maketitle
\tableofcontents

\section{Introduction}

In this paper we apply the Seidel representations for compact symplectic orbifolds, constructed in \cite{TW}, to give a description of quantum cohomology rings of symplectic toric orbifolds. The main idea in this work can be briefly summarized as follows. Let $(\X, \omega)$ be a compact symplectic orbifold. Denote by $Ham(\X,\omega)$ the ($2$-)group of Hamiltonian diffeomorphisms of $(\X,\omega)$. The Seidel representation of $(\X,\omega)$ is a group homomorphism $$\mathcal{S}: \pi_1(Ham(\X,\omega))\to QH_{orb}^*(\X,\omega)^\times$$ from the fundamental group of $Ham(\X, \omega)$ to the group $QH_{orb}^*(\X,\omega)^\times$ of multiplicatively invertible elements in the quantum cohomology ring of $(\X,\omega)$. Suppose there is a collection of loops $a_1, a_2,..., a_k$ in $Ham(\X,\omega)$ which compose to the identity loop $e$, namely
$$a_1\cdot a_2\cdot ...\cdot a_k=e,$$
Since $\mathcal{S}$ is a homomorphism, we have $$\mathcal{S}(a_1)*\mathcal{S}(a_2)*...* \mathcal{S}(a_k)=\mathcal{S}(e)=1.$$
This gives a relation in $QH_{orb}^*(\X,\omega)$. 

Suppose $(\X,\omega)$ is a compact symplectic toric orbifold. Then $\X$ can be defined by a combinatorial object called the {\em stacky fan} $\Sigfan=(\mathbf{N},\Sigma,\beta)$, see Section \ref{fanconstr} for a more detailed discussion. Every element $v\in \mathbf{N}$ in the lattice $\mathbf{N}$ determines a $\mathbb{C}^\times$-action on $\X$ and hence a loop in $Ham(\X,\omega)$. Let $S_v\in QH_{orb}^*(\X,\omega)$ denote the Seidel element corresponding to this loop. As discussed above, if $v_1, v_2,...,v_k\in \mathbf{N}$ are such that $v_1+v_2+...+v_k=0$ in $\mathbf{N}$. Then we have $$S_{v_1}* S_{v_2}*...*S_{v_k}=S_0=1.$$ 
This allows us to make use of additive relations in $\mathbf{N}$ to give a presentation of $QH_{orb}^*(\X,\omega)$. 

We now describe our results in more details. Let $\X$ be a compact symplectic toric orbifold associated with a labeled polytope\footnote{See Theorem \ref{thm:LT} for the correspondence between symplectic toric orbifolds and labeled polytopes.} $\Delta$ and let $\Sigfan=(\mathbf{N},\Sigma,\beta)$ be the stack fan associated to $\Delta$. Let $y_1,...,y_N\in\mathbf{N}$ be minimal generators of the rays in $\Sigma$. For each cone $\sigma$ in $\Sigma$, define 
$$\text{SBox}(\sigma):=\left\{b\in N\, |\, b=\sum_{y_i\in \sigma} a_i y_i, 0\leq a_i< 1\right\},$$
and let $\text{Gen}(\sigma)\subset \text{SBox}(\sigma)$ be the set of elements which cannot be generated by other lattice points in $\sigma$. The union $\text{Gen}(\Sigma):= \cup_{\sigma\in \Sigma}\text{Gen}(\sigma)$ is a finite set and we write $\text{Gen}(\Sigma)=\{y_{N+1},...,y_{M}\}$. For each $y_i$ we introduce a variable $X_i$. 
\begin{thm}[See Theorem \ref{thm:QH}]
There is an isomorphism of graded rings
$$\frac{\Lambda[X_1,...,X_M]}{Clos_{\mathfrak{v}_{T}}(\<\mathfrak{P}_{\xi} |\xi=1,...,n\> + SR_{\omega} + \mathcal{J}(\Sigma))} \simeq  QH_{orb}^*(\X,\Lambda),$$
where
\begin{enumerate}
\item $\Lambda$ is the Novikov ring in Definition \ref{dfn:Novikov};

\item The symbol $Clos_{\mathfrak{v}_{T}}(-)$ indicates that the closure with respect to the valuation $\mathfrak{v}_T$ in (\ref{dfn:valuation}).

\item $\mathfrak{P}_{\xi}\in \Lambda[X_{1},...,X_{M}], \xi =1,2,...,n$ are constructed in Theorem \ref{thm:QH};

\item $SR_\omega$ is the {\em quantum Stanley-Reisner ideal} in Definition \ref{dfn:QSR}.

\item $\mathcal{J}(\Sigfan)$ is the {\em cone ideal} defined in (\ref{eqn:cone_ideal}).

\end{enumerate}

\end{thm}

Quantum cohomology rings of toric orbifolds have been studied in various cases. The case of weighted projective lines is computed in \cite[Section 9]{AGV}. The quantum cohomology ring of an arbitrary weighted projective space is computed as a consequence of a mirror theorem in \cite{CCLT}. The case of orbifold projective lines with at most two cyclic orbifold points is computed in \cite{MiTs}. The quantum cohomology ring of a weak Fano toric orbifold is computed as a consequence of the mirror theorem in \cite{CCIT}. For an arbitrary toric orbifolds, the quantum cohomology ring is computed in \cite{GW} using the quantum Kirwan map \cite{W}. These previous works use various methods algebraic in nature. The method in this paper, based on Seidel representation, is symplectic.

The rest of this paper is organized as follows. Section \ref{sec:preliminary} contains reviews of preparatory materials including the basics of symplectic toric orbifolds, Chen-Ruan cohomology, and Hamiltonian loops. In Section \ref{sec:seidel} we review the results in \cite{TW} on the construction of Seidel representation for symplectic orbifolds. In Section \ref{seidelelement} we calculate Seidel elements arising from circle actions on a symplectic toric orbifold. This calculation is used in Section \ref{sec:results} to derive a presentation of quantum cohomology ring of a symplectic toric orbifold. In Section \ref{sec:Fano} we discuss the case of Fano toric orbifolds. 

Throughout this paper, $\X(\Sigfan)$ is the $2n$-dimensional compact symplectic toric orbifold associated with a labeled polytope $\Delta$ in the sense of Section \ref{polytopeconstr}. And $\Sigfan=(\mathbf{N},\Sigma,\beta)$ is the stack fan associated to $\Delta$.

\section*{Acknowledgement}
We thank Erkao Bao, Lev Borisov, Cheol-Hyun Cho, Conan Leung, and Yong-Geun Oh for valuable discussions.

\section{Preliminary on Toric orbifolds}\label{sec:preliminary}
In this section we review some basic constructions and facts about symplectic toric orbifolds. In this paper we only consider compact toric orbifolds whose generic stabilizer group is trivial, and we limit our discussion to that case.

\subsection{Construction via Stacky Fans}\label{fanconstr}
In algebraic geometry, toric orbifold are constructed using the combinatorial object called {\em stacky fans}. In this subsection we review this construction following \cite{BCS}.

By definition, a {\em stacky fan} consists of the following data $$\Sigfan=(\mathbf{N},\Sigma,\beta),$$ where 
\begin{enumerate}
\item
$\mathbf{N}$ is a finitely generated free abelian group of rank $n$; 
\item
$\Sigma\subset \mathbf{N}_\mathbb{Q}=\mathbf{N}\otimes_{\mathbb{Z}}\mathbb{Q}$ is a complete simplicial fan, with $\rho_1,...,\rho_N$ being its $1$-dimensional cones; 
\item
$\beta: \mathbb{Z}^{N}\to \mathbf{N}$ is a map determined by the elements $\{b_{1},\cdots,b_{N}\}$ in $\mathbf{N}$ satisfing that $b_i\in \rho_i$.  More precisely, let $e_1,...,e_N\in \mathbb{Z}^N$ be the standard basis, then $\beta(e_i):=b_i$. 
\end{enumerate}

We assume that $\beta$ has finite cokernel, and $\{b_1,...,b_N\}\subset \mathbf{N} \subset \mathbf{N}_\mathbb{Q}$  generate the simplicial fan $\Sigma$. 

The toric orbifold (also known as toric Deligne-Mumford stack) $\XSigma$  associated to $\Sigfan$ is defined to be the following quotient stack 
\begin{equation}\label{def_toric_orbifold}
\XSigma:=[Z/G],
\end{equation}
whose definition may be explained as follows:
\begin{enumerate}
\item
 $Z$ is the open subvariety $\mathbb{C}^{N}\setminus\mathbb{V}(J_{\Sigma})$. Here $J_{\Sigma}$ is the irrelevant ideal of the fan, defined as follows: let $\mathbb{C}[z_1,...,z_N]$ be the coordinate ring of $\mathbb{C}^N$, then $J_\Sigma$ is the ideal of $\mathbb{C}[z_1,...,z_N]$ generated by the monomials $\prod_{\sigma_i\nsubseteq \sigma} z_i$ where $\sigma$ run through all cones in $\Sigma$. 
\item 
 $G$ is an algebraic torus defined by $G=\text{Hom}_{\mathbb{Z}}(\mathbf{N}^{\vee},\mathbb{C}^{*})$, where $\mathbf{N}^\vee$ appears in the Gale dual $\beta^{\vee}: \mathbb{Z}^{N}\to \mathbf{N}^{\vee}$ of $\beta$ (see \cite{BCS}). The $G$-action on $Z$ is given by a group homomorphism $\alpha: G \to (\mathbb{C}^{*})^{N}$ obtained by applying the functor  $\text{Hom}_{\mathbb{Z}}(-,\mathbb{C}^{*})$  to the Gale dual $\beta^{\vee}: \mathbb{Z}^{N}\to N^{\vee}$ of $\beta$.  
 
 \end{enumerate}
The quotient (\ref{def_toric_orbifold}) may be taken in different categories. In this paper we deal with symplectic toric orbifolds. For this reason we consider the quotient as a {\em differentiable stack}. Indeed the $G$-action on $Z$ defines a groupoid $G\ltimes Z:=(G\times Z\rightrightarrows Z)$ where the source and target maps are the projection to the second factor and the $G$-action respectively. It is easy to see that $G\ltimes Z$ is a proper Lie groupoid and the quotient stack (\ref{def_toric_orbifold}) is the differentiable stack associated to this Lie groupoid. 

The toric orbifold $\XSigma$ has a collection of naturally defined \'{e}tale charts which we describe. Let $\sigma\in \Sigma$ be a $k$-dimensional cone generated by $\{b_{i_{1}},...,b_{i_{k}} \}$. Define the open subset $U(\sigma)\subset \com^{N}$ as
$U(\sigma) = \{(z_{1},...,z_{N}) \in \com^{N} | z_{j}\neq 0 \ \forall j \notin \{i_{1},... ,i_{k}\}\}$. Then $U(\sigma)\subset U(\sigma')$ if $\sigma$ is contained in $\sigma'$, and $\{ U(\sigma)|\sigma\in \Sigma,\ dim\ \sigma=n\}$ is an open cover of $Z$. 

Each $n$-dimensional cone $\sigma$ induces an orbifold chart of $\XSigma$ which covers $U_{\sigma}:=U(\sigma)/G$. More explicitly, let $\mathbf{N}_{\sigma}$ be the sublattice of $\mathbf{N}$ generated by $\{b_{i_{1}},...,b_{i_{n}}\}$, $N_{\sigma}^{*}$ be the dual lattice of $\mathbf{N}_{\sigma}$, and $\{u_{j}\}_{j=1}^{n}$ be the dual basis of $\mathbf{N}_{\sigma}^{*}$ so that $\langle b_{i_{k}}, u_{j}\rangle =\delta_{k,j}$,
then we have a map from $U_{\sigma}$ to $\com^{n}$ by 
$$w_{j} =z_{1}^{\langle b_{1},u_{j}\rangle} \cdot ...\cdot z_{N}^{\langle b_{N},u_{j}\rangle},\ \ \ \ j=1,...,n.$$
The image of $U(\sigma)$ under this map, denoted as $V_{\sigma}$, carries a group action by $G_{\sigma}:=\mathbf{N}/\mathbf{N}_{\sigma}$:
$$g\cdot w_{j}=e^{2\pi i\langle \tilde{g},u_{j}\rangle}w_{j}\ \ \ \ \text{for}\ \tilde{g}\in \mathbf{N}\ \text{lifting}\ g\in \mathbf{N}/\mathbf{N}_{\sigma}, \ \ \ j=1,...,n.$$
Then $G_{\sigma}\ltimes V_{\sigma}$ defines an orbifold chart over $U_{\sigma}$.

Now if $\tau$ is a $k$-dimensional cone contained in an $n$-dimensional cone $\sigma$, then the orbit $O_{\tau}$ determined by $\tau$ has a neighborhood $U_{\tau}$, the orbifold chart $G_{\sigma}\ltimes V_{\sigma}$ restricted to $U_{\tau}$ defines an orbifold chart of $U_{\tau}$. Note that this chart is not effective. After reduction, we get an orbifold chart $G_{\tau}\ltimes V_{\tau}$, where $G_{\tau}=(N_{\tau}\otimes_{\inte} \ration)\cap \mathbf{N}\ \ /\mathbf{N}_{\tau}$, $V_{\tau}$ is an open set of $\com^{n}$.

Given an orbifold $\X$ one can consider its inertia orbifold $I\X:= \X\times_{\X\times \X} \X$, where the fiber product is taken over the diagonal map $\X\to \X\times \X$. In the toric case we can give a more combinatorial description of the inertia orbifolds, following \cite{BCS}. Let $\X(\Sigfan)$ be the toric orbifold defined by the stacky fan $\Sigfan:=(\mathbf{N},\Sigma,\beta)$. For a cone $\tau\in \Sigma$, define 
$$\text{Box}(\tau):=\left\{v\in \mathbf{N}\, |\, v=\sum_{b_i\in \tau} r_i b_i, 0\leq r_i< 1\right\},$$
and set $\text{Box}(\Sigfan):=\bigcup_{\tau\in \Sigma} \text{Box}(\tau)$.  Then we have:
\begin{enumerate}
\item The set $\text{Box}(\Sigfan)$ indexes the components of the inertia orbifold of $\X(\Sigfan)$: 
$$I\XSigma=\sqcup_{v\in \text{Box}(\Sigfan)}\X_{(v)}.$$ 
Here $\sigma(v)\in \Sigma$ is the minimal cone containing $v$ and $\X_{(v)}=\X(\Sigfan/\sigma(v))$ is the toric orbifold associated to the stacky fan $\Sigfan/\sigma(v)$ defined in \cite[Section 4]{BCS}. Each component $\X_{(v)}$ is called a {\em twisted sector}, and $\X_{(0)}\simeq \X$ is called the trivial twisted sector or untwisted sector. 
\item There is a natural involution $\mathcal{I}: I\X\to I\X$ defined by $\mathcal{I}((x,v)):=(x, v^{-1})$, where $v\in\text{Box}(\tau)$ and $v^{-1}$ is the unique element  in $\text{Box}(\tau)$ such that $v^{-1}+v \in \mathbf{N}_{\tau}$. Note that $v^{-1}$ and $-v$ are different, the later does not lie in $\text{Box}(\tau)$.
\end{enumerate}

\subsection{Symplectic Toric Orbifolds via Labelled Polytopes}\label{polytopeconstr}
Intrinsically a symplectic toric orbifold is a symplectic orbifold with a Hamiltonian action by a half-dimensional compact torus. In \cite{LT} symplectic toric orbifolds are classified by the combinatorial objects called labelled polytopes. This is a generalization of the classification of symplectic toric manifolds by Delzant polytopes. In this subsection we review the relation between symplectic toric orbifolds, labelled polytopes, and the stacky fan construction in the previous subsection. 

Let $(\X,\omega)$ be a $2n$-dimensional (compact) symplectic orbifold. There is a Hamiltonian $T^n$-action on $\mathcal{X}$ with moment map $$\Phi: \X\to \textbf{t}^*,$$  where $\textbf{t}$ is the Lie algebra of $T^{n}$ with a lattice $\textbf{l}$ and $\textbf{t}^*$ is the dual vector space of $\textbf{t}$. 
The image $\Phi(\X)$ of the moment map is a rational simple convex polytope which is defined as below:

\begin{defn}
A convex polytope $\Delta\subset \textbf{t}^*$ is rational if
$$
\Delta= \bigcap_{i=1}^N\{\alpha\in\textbf{t}^*| \<\alpha,n_i\>\le \lambda_i\}
$$
for some $n_i\in \textbf{l}$ and $\lambda_i\in \real$. 

A (closed) facet is a face of $\Delta$ of codimension one in $\Delta$. An open facet is the relative interior of a facet. An n dimensional polytope is simple if exactly $n$ facets meet at every vertex. 

A convex rational simple polytope $\Delta$ such that $dim\Delta=dim\ \textbf{t}$, plus a positive integer attached to each open facet, is called a {\em labeled polytope}. Two labeled polytopes are isomorphic if one can be mapped to the other by a translation and the corresponding open facets have the same integer labels.
\end{defn}

The following result is due to Lerman and Tolman.

\begin{thm}[\cite{LT}, Theorem 1.5]\label{thm:LT}
\hfill
\begin{enumerate}
\item A compact symplectic toric orbifold $(\mathcal{X},\omega,T,\Phi)$ naturally corresponds to a labeled polytope, namely the image of the moment map $\Phi(\mathcal{X})$, which is a rational simple polytope. For every open facet $\intF$ of $\Phi(\mathcal{X})$ there exists a positive integer $n_{\intF}$ such that the structure group of every $x\in \Phi^{-1}(\intF)$ is $\inte / n_{\intF}\inte$.
\item Two compact symplectic toric orbifolds are isomorphic if and only if their associated labeled polytopes are isomorphic.
\item Every labeled polytope can be realized as the image of the moment map for some compact symplectic toric orbifold.
\end{enumerate}
\end{thm}

Now, we consider a symplectic toric orbifold $(\mathcal{X},\omega,T,\Phi)$ determined by a polytope $\Delta \subset \textbf{t}^*$ with labels $m_i$ on facets $F_i$. Let $y_i$ be the primitive outward normal vectors of $F_i$. For each face $F$ of $\Delta$, let $\sigma_{F}\in \Sigma$ be the cone determined by the collection of vectors $\{y_{i}|F\subset F_{i}\}$. Define 
\begin{itemize}
\item $\mathbf{N}=\textbf{l}=\{\sum_{i} k_{i}y_{i}|k_{i}\in \inte\}\subset \textbf{t}$, 
\item $\Sigma=\{\sigma_{F}|F\ \text{a face of}\ \Delta\}$,
\item and $b_{i}=m_{i} y_{i}$, i.e. $\beta(e_i)=m_i y_i$.
\end{itemize}
Then the data $(\mathbf{N}, \Sigma, \{b_i\}_{i=1}^{N})$ give a stacky fan. In this way every labelled polytope gives rise to a stacky fan.

Every corner (0-dimensional face) $C$ of $\Delta$ determines a point $x\in\X$ which is fixed by the torus action. 
Let $\sigma_{C}$ be the cone corresponding to $C$, then $\sigma_{C}$ determines an orbifold chart $G_{\sigma(C)}\ltimes V_{\sigma(C)}$ covering $x$ as in Section \ref{fanconstr}. Let $\{u^{C}_{j}|j=1,...,n\}\subset \textbf{t}^{*}$ be the dual basis of $\{y_{i}|y_{i}\in \sigma(C),\ 1\le j\le N\}$. Then the following lemma is an orbifold version of Theorem 3.1.2 in \cite{dS}.
\begin{lem}\label{equivDarboux}
The point $x$  is covered by an orbifold chart $G_{C}\ltimes V_{C}$ such that
\begin{itemize}
\item $G_{C}\ltimes V_{C}$ is isomorphic to $G_{\sigma(C)}\ltimes V_{\sigma(C)}$;
\item the symplectic form can be written as 
$$\omega=\sum_{j=1}^{n}dp_{j}\wedge dq_{j}\ \ \ with\ w_{j}=p_{j}+iq_{j};$$ 
\item The moment map can be written as 
$$\Phi(p_{1},...,p_{n};q_{1},...,q_{n})=\Phi(x)-\sum_{j=1}^{n} u^{C}_{j}(p_{j}^{2}+q_{j}^{2}).$$
\end{itemize}
\end{lem}

\subsection{Chen-Ruan orbifold cohomology of toric orbifolds}\label{toricQH}
In this subsection we describe the calculation of the Chen-Ruan orbifold cohomology ring of toric orbifolds, following \cite{BCS}\footnote{Strictly speaking what is computed in \cite{BCS} is the orbifold Chow ring. However the computation for Chen-Ruan orbifold cohomology ring is identical.}. 

Let $\XSigma$ be a compact symplectic toric orbifold given by a stacky fan $\Sigfan=(\mathbf{N},\Sigma,\beta)$. Recall that the Chen-Ruan orbifold cohomology $H_{CR}^{*}(\XSigma, \ration)$ is defined as the direct sum of the cohomology groups of its inertia orbifold with a shifted grading:  For $v=\sum_{b_{i}\in \sigma(v)} r_{i} b_{i}\in  \text{Box}(\Sigma)$. The corresponding twisted sector $\X_{(v)}$ is associated with a number $\iota_{v}=\sum_{b_i\in \sigma(v)} r_{i}$ called age or degree shifting number. We have
\begin{equation*}
H_{CR}^{*}(\XSigma,\ration)=\oplus_{v\in \text{Box}(\Sigfan)} H^{*-2\iota_{v}}(\X_{(v)},\ration).
\end{equation*}
There is a product on $H_{CR}^{*}(\XSigma,\ration)$, called the Chen-Ruan cup product, which is defined using genus $0$ degree $0$ three-point Gromov-Witten invariants of $\XSigma$. We refer to \cite{CR1} for more details of the definition of this product. This construction makes $H_{CR}^{*}(\XSigma,\ration)$ into a graded algebra.
 
Let $\mathbf{M}=\mathbf{N}^*=\text{Hom}_\mathbb{Z}(\mathbf{N}, \mathbb{Z})$ be the dual of $\mathbf{N}$. Let $\ration[\mathbf{N}]^{\Sigfan}$ be the group ring of $\mathbf{N}$, i.e. $\ration[\mathbf{N}]^{\Sigfan}:=\bigoplus_{c\in \mathbf{N}}\ration \lambda^{v}$, $\lambda$ is the formal variable. A $\ration$-grading on $\ration[\mathbf{N}]^{\Sigfan}$ is defined as follows. For $v\in \mathbf{N}$, if $v =\sum_{b_i\in \sigma(v)}r_i b_i$ where $\sigma(v)$ is the minimal cone in $\Sigma$ containing ${v}$ and $r_i$ are nonnegative rational numbers, then we define 
\begin{equation}\label{defn:age-grading}
\text{deg}\, (\lambda^v):=\sum_{b_i\in \sigma(v)}r_i.
\end{equation}

Define the following multiplication on $\ration[N]^{\Sigfan}$:
\begin{equation}\label{product}
\lambda^{v_{1}}\cdot \lambda^{v_{2}}:=\begin{cases}\lambda^{v_{1}+v_{2}}&\text{if
there is a cone}~ \sigma\in\Sigma ~\text{such that}~ {v}_{1}, {v}_{2}\in\sigma\,,\\
0&\text{otherwise}\,.\end{cases}
\end{equation}
Let $\mathcal{I}(\Sigfan)$ be the ideal in
$\ration[\mathbf{N}]^{\Sigfan}$  generated by the elements $\sum_{i=1}^{n}\theta(b_{i})\lambda^{b_{i}}, \theta\in \mathbf{M}$. Then by \cite[Theorem 1.1]{BCS}, there is an isomorphism of $\ration$-graded algebras:
\begin{equation}\label{Horb_presentation}
H_{CR}^{*}\left(\XSigma, \ration \right)\cong \frac{\ration[\mathbf{N}]^{\Sigfan}}{\mathcal{I}(\Sigfan)}.
\end{equation}


%
%

We will rewrite (\ref{Horb_presentation}) in terms of a quotient of polynomial ring by some ideal. Denote $X_i=\lambda^{y_i}$ for $i=1,..., N$.

Let 
$$\text{SBox}(\sigma):=\left\{b\in N\, |\, b=\sum_{y_i\in \sigma} a_i y_i, 0\leq a_i< 1\right\},$$
and 
$\text{Gen}(\sigma)\subset \text{SBox}(\sigma) $ 
be the set of {\em minimal elements}, which means these elements cannot be generated by other lattice points in the cone $\sigma$. Obviously if $\sigma\subset\sigma'$ then $\text{Gen}(\sigma)\subset \text{Gen}(\sigma')$.

Then $\text{Gen}(\Sigma):= \cup_{\sigma\in \Sigma}\text{Gen}(\sigma)$ is a finite set. For convenience, we denote $\text{Gen}(\Sigma)=\{y_{N+1},...,y_{M}\}$, and define $X_{N+1}=\lambda^{y_{N+1}},...,X_{M}=\lambda^{y_{M}}$. We call $I\subset\{1,...,M\}$ a {\em generalized primitive collection} if 
\begin{itemize}
\item $\{y_{i}|i\in I\}$ is not contained in a cone,
\item any proper subset of $\{y_{i}|i\in I\}$ is contained in some cone.
\end{itemize}
Note that when a generalized primitive collection $I$ is a subset of  $\{1,...N\}$, then it is a primitive collection in the sense of \cite{Ba}. Denote by $\mathcal{GP}$ the set of all generalized primitive collections.

For top-dimensional cones $\sigma_{j}$, $j=1,...,N$, 
define an ideal 
$$\mathcal{J}(\sigma_{j}):=\left\langle\prod_{t_{i}>0, y_{i}\in \sigma_{j}} X_{i}^{t_{i}}-\prod_{t_{i}<0, y_{i}\in \sigma_{j}} X_{i}^{-t_{i}}|\sum_{y_{i}\in \sigma_{j}} t_{i}y_{i}=0,\ t_{i}\in \inte, \vec{t}\neq \vec{0}\right\rangle$$
Set 
\begin{equation}\label{eqn:cone_ideal}
\mathcal{J}(\Sigfan)=\sum_{j=1}^{N} \mathcal{J}(\sigma_{j}).
\end{equation}
 We call $\mathcal{J}(\Sigfan)$ the {\em cone ideal}, and elements in $\mathcal{J}(\Sigfan)$ the cone relations.

Now the Chen-Ruan cohomology can be rewritten as: 
$$H^*_{CR}(\mathcal{X},\ration)\cong\frac{\ration[[X_1,...,X_M]]}{\mathcal{I}(\Sigfan)+ \<\Pi_{i\in I} X_i: I\in\mathcal{GP} \>+ \mathcal{J}(\Sigfan)}.$$
Moreover, since we know $H^*_{CR}(\mathcal{X},\ration)$ is finite dimensional, a monomial in $X_{1},...,X_{M}$ in the right side vanishs if its degree is large enough. Thus we have the following lemma:
\begin{lem}\label{CRring}
$$H^*_{CR}(\mathcal{X},\ration)\cong\frac{\ration[X_1,...,X_M]}{\mathcal{I}(\Sigfan)+ \<\Pi_{i\in I} X_i: I\in\mathcal{GP} \>+ \mathcal{J}(\Sigfan)}.$$
\end{lem}

\begin{rmk}
The grading of $\ration[X_1,...X_M]$ is not the usual grading of polynomial ring.
\end{rmk}

\subsection{Hamiltonian Loops of Toric Orbifolds}\label{hamloop}
Now, we consider a symplectic orbifold $(\mathcal{X},\omega,T,\Phi)$ determined by a polytope $\Delta \subset \textbf{t}^*$ with labels $m_i$ on facets $F_i$. In this case, we give an explicit description of Hamiltonian loops defined in \cite[Section 2.3]{TW}.

Every integer vector $v\in \textbf{l}\subset \textbf{t}$ determines a Hamiltonian function $$H_{v}: \X \to \real,\quad H_{v}(x):=\langle\Phi(x),v\rangle,$$ whose flow determines a Hamiltonian loop $\phi_{v}: I\times \X\to \X$ with some 2-morphism $\phi_{v}|_{\{0\}\times \X}\Rightarrow \phi_{v}|_{\{1\}\times \X}$. Let $\phi_{v_{1}}$, $\phi_{v_{2}}$ be the Hamiltonian loops determined by $v, v'\in \textbf{l}$, then the composition product of the two loop $\phi_{v_{1}}\cdot_{cp}\phi_{v_{2}}$ is generated by $$H(x)=\langle\Phi(x),v_{1}\rangle+\langle\Phi(\phi_{v_{1}}(-t,x)),v_{2}\rangle=\langle\Phi(x),v_{1}\rangle+\langle\Phi(x),v_{2}\rangle,$$ where the second equality holds because the moment map is invariant under the torus action. So $H(x)=\langle\Phi(x),v_{1}+v_{2}\rangle$, and $\phi_{v_{1}}\cdot_{cp}\phi_{v_{2}}$ is determined by the vector $v_{1}+v_{2}\in \textbf{l}$.

Now we give a local description of Hamiltonian functions and Hamiltonian loops. Let $C$ be a corner of $\Delta$, then it determines  a chart $G_{C}\ltimes V_{C}$ centered at the fixed point $x_{C}$ corresponding to $C$ as in Lemma \ref{equivDarboux}. On this chart, the Hamiltonian loop generated by a vector $v\in \textbf{l}$ can be represented by a groupoid morphism 
\begin{eqnarray*}
\gamma & & :[0,1]\times G_{C}\ltimes V_{C}\to  G_{C}\ltimes V_{C},\\
\gamma_{0}(t,\vec{w}) & & =D(tv)\vec{w},\\
\gamma_{0}(t,\vec{w}\xrightarrow{h} h\cdot \vec{w}) & & =(D(tv)\vec{w} \xrightarrow{h} h\cdot D(tv)\vec{w}).
\end{eqnarray*}
where  $\vec{w}=(w_{1},...,w_{n})$, and $D(tv)$ is the diagonal matrix $diag(e^{-i2\pi\langle u_{1},tv\rangle},...,e^{-i2\pi\langle u_{n},tv\rangle})$.

\begin{lem}\label{loopatlas}
The natural transformation $\alpha_{v}=(\gamma|_{\{1\}\times G_{C}\ltimes V_{C}} \Rightarrow Id_{G_{C}\ltimes V_{C}})$ is given by
\begin{eqnarray*}
\alpha_{v}  : V_{C} & & \to\ \  G_{C},\\
 x & & \to \ \ D(-v):=diag(e^{-i2\pi\langle u_{1},v\rangle},...,e^{-i2\pi\langle u_{n},v\rangle}).
\end{eqnarray*}
\end{lem}

\begin{example}
Consider the weighted projective line $\com P(1,k):=\com^{2}\setminus \{0\}/\com^{*}$, where $\com^{*}$ acts on $\com^{2}\setminus \{0\}$ by $z\cdot (w_{1},w_{2})=(zw_{1},z^{k}w_{2})$. The corresponding moment polytope is a line segment with one end labeled with $k$.
\begin{figure}[htb!]\label{fig_CP1k}
\centering
\begin{tikzpicture}
 \draw (-1,0) node[left] {1} -- (1,0) node[right] {k};
 \draw (0, 0) node[above] {0};
 \draw (0, -0.2) node[below] {Labelled Polytope};
 \draw[-triangle 90] (5,0) node[above] {0} -- (4,0) node[below] {-1};
  \draw (6, -0.2) node[below] {Stacky Fan};
\draw (5,0) --  (6.2,0);        
\draw[dashed] (6.2,0) --  (6.8,0);        
\draw[-triangle 90] (6.8,0) --  (8,0) node[below] {k};     
\draw (6,0) node[above] {$v$};    
 \fill[black, fill opacity=0] (0,0) circle (0.03);      
 \fill[black, fill opacity=0] (-1,0) circle (0.05);        
 \fill[black, fill opacity=0] (1,0) circle (0.05);        
 \fill[black, fill opacity=0] (4,0) circle (0.05);  
  \fill[black, fill opacity=0] (5,0) circle (0.05);        
 \fill[black, fill opacity=0] (6,0) circle (0.05);
 \fill[black, fill opacity=0] (7,0) circle (0.05);        
 \fill[black, fill opacity=0] (8,0) circle (0.05);  
\end{tikzpicture} 
\caption{Labelled Moment Polytope and Stacky Fan of $\com P(1,k)$.}
\end{figure}
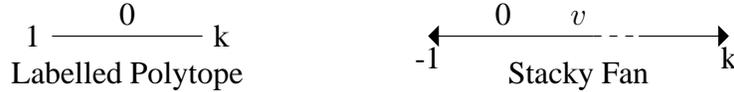

The point $[0,1]$ has non-trivial isotropy group $\inte_{k}$. There is a groupoid chart $\inte_{k}\ltimes \com$ covering $|\com(1,k)|\setminus \{[1,0]\}$, such that the Hamiltonian loop determined by the vector $v$ in Figure 1 is given by 
\begin{eqnarray*}
\gamma & & :[0,1]\times \inte_{k}\ltimes \com\to  \inte_{k}\ltimes \com,\\
\gamma_{0}(t,w) & & =e^{\frac{2\pi}{k} it}w,\\
\gamma_{0}(t,w\xrightarrow{h} h\cdot w) & & =(e^{\frac{2\pi}{k} it} w \xrightarrow{h} e^{\frac{2\pi}{k} it} w).
\end{eqnarray*}

\end{example}

\subsection{Hamiltonian Orbifiber Bundles over Sphere}\label{orbifiberSec}
Given a Hamiltonian loop, one can construct a Hamiltonian orbifiber bundle over $S^{2}$ as explained in \cite[Section 2.5]{TW}. In this section we will consider the Hamiltonian orbifiber bundles determined by $v\in \text{Gen}(\Sigma)$, because later we shall compute Seidel element for these corresponding loops. 

Denote by $H_{v}=\langle \Phi,v\rangle:\ \X\to \real$ the Hamiltonian function, $\phi_{v}$ the Hamiltonian loop determined by $H_{v}$, and $\E_{v}$ the corresponding orbifiber bundle. If $H_{v}$ attends the maximum at $x\in \X$, then $x$ is fixed by the Hamiltonian loop, thus $|x|\in |\X|$ determines a section $s_{x}$ of the topological fiber bundle $|\E_{v}|\to S^{2}$ underlying $\E_{v}\to S^{2}$. The main purpose of this section is to study the properties of sectional morphisms in $\E_{v}$ which lift $s_{x}$. 

It is easy to see that if $H_{v}$ attends the maximum at $x$, then $\Phi(x)\in F_{v}$, where $F_{v}$ is the face of $\Delta$ determined by the minimal cone containing $v$ according to the cone-face correspondence in Section \ref{polytopeconstr}.

Let $C$ be a corner of $\Delta$ such that $C\subset F_{v}$, and let $x_{C}$ be the fixed point whose image under the moment map is $C$. Recall from Section \ref{polytopeconstr} that there is a chart $G_{C}\ltimes V_{C}$ centered at $x_{C}$. The orbifold $[G_{C}\ltimes V_{C}]$ can be regarded as an open suborbifold of $\X$ containing $x$ and $x_{C}$. The Hamiltonian loop $\gamma_{v}$ restricts to a Hamiltonian loop of $[G_{C}\ltimes V_{C}]$. Then we have an orbifiber bundle $\E_{v}(C)\to S^{2}$ determined by $\gamma_{v}$ with fiber $[G_{C}\ltimes V_{C}]$. Then $\E_{v}(C)$ is an open suborbifold of $\E_{v}$. We will study sectional morphisms in $\E_{v}$ lifting $s_{x}$ inside $\E_{v}(C)$ since it has a nice groupoid chart.

Recall from the Section \ref{hamloop} that the Hamiltonian loop $\phi_{v}$ restricts to a Hamiltonian loop of $[G_{C}\ltimes V_{C}]$ represented by:
\begin{eqnarray*}
\gamma & & :[0,1]\times G_{C}\ltimes V_{C}\to  G_{C}\ltimes V_{C},\\
\gamma_{0}(t,\vec{w}) & & =D(tv)\vec{w},\\
\gamma_{0}(t,\vec{w}\xrightarrow{h} h\cdot \vec{w}) & & =(D(tv)\vec{w} \xrightarrow{h} h\cdot D(tv)\vec{w}).
\end{eqnarray*}
Let 
$$U_{L}:= \{e^{2\pi it}|t\in (0,1)\}=
\begin{aligned}
\begin{tikzpicture}
 \def\R{1} 
 \draw[thick] (0.5*\R,0) arc (7:355:0.5*\R);
\filldraw[fill=white, draw=black, fill opacity=0] (0.5*\R, 0) circle (0.06);
\end{tikzpicture} 
\end{aligned}
, \ \ \ \ \ \ U_{R}:=\{e^{2\pi it}|t\in (-\frac{1}{2},\frac{1}{2})\}=
\begin{aligned}\begin{tikzpicture}
 \def\R{1} 
  \draw[thick] (-0.5*\R,0) arc (-175:175:0.5*\R);
\filldraw[fill=white, draw=black, fill opacity=0] (-0.5*\R, 0) circle (0.06);
\end{tikzpicture}\end{aligned}.$$

Define $U_{S^{1}}$ by
\begin{eqnarray*}
 Ob(U_{S^{1}}) & = & U_{L}\sqcup U_{R},\\
 Mor(U_{S^{1}}) & = & U_{L}\times_{S^{1}} U_{L}\ \sqcup\ U_{R}\times_{S^{1}} U_{R}\ \sqcup\ U_{L}\times_{S^{1}} U_{R}\ \sqcup\ U_{R}\times_{S^{1}} U_{L}.
\end{eqnarray*}
Note that 
 \begin{eqnarray*}
\begin{aligned} U_{L}\times_{S^{1}} U_{L} \end{aligned} =  \begin{aligned}\begin{tikzpicture}
 \def\R{1} 
 \draw[thick] (0.5*\R,0) arc (7:355:0.5*\R);
\filldraw[fill=white, draw=black, fill opacity=0] (0.5*\R, 0) circle (0.06);
\end{tikzpicture}\end{aligned} & & , \ \  \
 U_{R}\times_{S^{1}} U_{R}  = \begin{aligned} \begin{tikzpicture}
 \def\R{1} 
   \draw[thick] (-0.5*\R,0) arc (-175:175:0.5*\R);
\filldraw[fill=white, draw=black, fill opacity=0] (-0.5*\R, 0) circle (0.06);
\end{tikzpicture}\end{aligned},\\
U_{L}\times_{S^{1}} U_{R}  =   \begin{aligned}\begin{tikzpicture}
 \def\R{1} 
 \draw[thick] (0.5*\R,0) arc (0:180:0.5*\R);
\end{tikzpicture} \end{aligned} \  \ \ \sqcup\ \ \ \  \begin{aligned}\begin{tikzpicture}
 \def\R{1} 
 \draw[thick] (-0.5*\R,0) arc (180:360:0.5*\R);
\end{tikzpicture}\end{aligned} & & , 
\ \ \ \ \ \ 
U_{R}\times_{S^{1}} U_{L}  = \begin{aligned} \begin{tikzpicture}
 \def\R{1} 
 \draw[thick] (0.5*\R,0) arc (0:180:0.5*\R);
\end{tikzpicture} \end{aligned} \ \ \ \sqcup \ \ \  \begin{aligned} \begin{tikzpicture}
 \def\R{1} 
 \draw[thick] (-0.5*\R,0) arc (180:360:0.5*\R);
\end{tikzpicture}\end{aligned} .
\end{eqnarray*}

We denote $(e^{2\pi i t})_{*}\in U_{*}$ and $g=(e^{2\pi i t})_{*}\to (e^{2\pi i t'})_{\bullet}\in U_{*}\times_{S^{1}} U_{\bullet} $,  for $*, \bullet=R, L$, where
\begin{itemize}
\item  $t'=t-1$, if $g\in \begin{aligned} \begin{tikzpicture}
 \def\R{1} 
 \draw[thick] (-0.5*\R,0) arc (180:360:0.5*\R);
\end{tikzpicture}\end{aligned} \subset U_{L}\times_{S^{1}} U_{R} $;
\item $t'=t+1$, if $g\in \begin{aligned} \begin{tikzpicture}
 \def\R{1} 
 \draw[thick] (-0.5*\R,0) arc (180:360:0.5*\R);
\end{tikzpicture}\end{aligned} \subset U_{R}\times_{S^{1}} U_{L} $;
\item $t'=t$, otherwise.
\end{itemize}

Then define a groupoid morphism $\tilde{\gamma}: U_{S^{1}}\times G_{C}\ltimes V_{C}\to G_{C}\ltimes V_{C}$ by:
\begin{eqnarray*}
\tilde{\gamma}_{0}((e^{2\pi i t})_{L},\vec{w}) & & :=\gamma_{0}(t,\vec{w})= D(tv)\vec{w},  \ \ \ \ \ \ \ \ t\in(0,1),\ \ \ (e^{ 2\pi i t})_{L}\in U_{L};\\
\tilde{\gamma}_{0}((e^{2\pi i t})_{R},\vec{w}) & & :=D(tv)\vec{w} , \ \ \ \ \ \ \ \ \ \ \ \ \ \ \ \ \ \ \ \ \ \ \ \ \ \ t\in(-\frac{1}{2},\frac{1}{2}),\ (e^{ 2\pi i t})_{R}\in U_{R};\\
\tilde{\gamma}_{1}((e^{2\pi i t})_{L}\to (e^{ 2\pi i t})_{L}, \vec{w} \xrightarrow{h}  h\cdot\vec{w}) & & :=D(tv)\vec{w} \xrightarrow{h}  h\cdot D(tv)\vec{w}, \ \ \ t\in(0,1);\\
\tilde{\gamma}_{1}((e^{2\pi i t})_{R}\to (e^{ 2\pi i t})_{R}, \vec{w} \xrightarrow{h}  h\cdot\vec{w}) & & :=D(tv)\vec{w} \xrightarrow{h}  h\cdot D(tv)\vec{w},\ \ \ \ t\in(-\frac{1}{2},\frac{1}{2});\\
\tilde{\gamma}_{1}((e^{2\pi i t})_{L}\to (e^{ 2\pi i t})_{R}, \vec{w} \xrightarrow{h}  h\cdot\vec{w}) & & :=D(tv)\vec{w} \xrightarrow{h}  h\cdot D(tv)\vec{w}, \ \ \ t\in(0,\frac{1}{2});\\
\tilde{\gamma}_{1}((e^{2\pi i t})_{R}\to (e^{ 2\pi i t})_{L}, \vec{w} \xrightarrow{h}  h\cdot\vec{w}) & & :=D(tv)\vec{w} \xrightarrow{h}  h\cdot D(tv)\vec{w},\ \ \ \ t\in(0,\frac{1}{2});\\
\tilde{\gamma}_{1}((e^{2\pi i t})_{L}\to (e^{ 2\pi i (t-1)}))_{R}, \vec{w} \xrightarrow{h}  h\cdot\vec{w}) & & :=D(tv)\vec{w} \xrightarrow{h\cdot D(-v)}  h\cdot D((t-1)v)\vec{w}, \ \ \ t\in (\frac{1}{2},1);\\
\tilde{\gamma}_{1}((e^{2\pi i t})_{R}\to (e^{ 2\pi i (t+1)}))_{L}, \vec{w} \xrightarrow{h}  h\cdot\vec{w}) & & :=D(tv)\vec{w} \xrightarrow{h\cdot D(v)}  h\cdot D((t+1)v)\vec{w}, \ \ \ t\in(-\frac{1}{2},0);\\
\end{eqnarray*}

Consider the following groupoid chart of $S^{2}$: 
\begin{eqnarray*}
 Ob(U_{S^{2}}) & & :=\begin{aligned}
 \begin{tikzpicture}
 \def\R{1} 
\filldraw[gray, ultra nearly transparent] (0.4*\R,-0.3*\R) arc (-asin(0.6):180+asin(0.6):0.5*\R) arc (180:0:0.4 and 0.06)-- cycle;
\fill[gray, semitransparent] (0,-0.3*\R) ellipse (0.4 and 0.05); 
\shadedraw[shading=ball,ball color=red](0.4*\R,-0.3*\R) arc (-asin(0.6):-180+asin(0.6):0.5*\R) arc (180:360:0.4 and 0.06) -- cycle;
\end{tikzpicture} 
\end{aligned}
\sqcup
\begin{aligned}
\begin{tikzpicture}
 \def\R{1} 
 \shadedraw[shading=ball,ball color=red, white] (0,0) circle (0.5*\R);
 \filldraw[fill=white, draw=black, fill opacity=0] (0,0.5*\R) circle (0.06);
 \filldraw[fill=white, draw=black, fill opacity=0] (0,-0.5*\R) circle (0.06);
\draw[thick, double] (0,0.5*\R) arc (90:-90:0.5*\R);
\end{tikzpicture} 
\end{aligned}
\sqcup
\begin{aligned}
\begin{tikzpicture}
 \def\R{1} 
 \shadedraw[shading=ball,ball color=red, white] (0,0) circle (0.5*\R);
  \filldraw[fill=white, draw=black, fill opacity=0] (0,0.5*\R) circle (0.06);
 \filldraw[fill=white, draw=black, fill opacity=0] (0,-0.5*\R) circle (0.06);
 \draw[thick, double] (0,-0.5*\R) arc (-90:-270:0.5*\R);
\end{tikzpicture} 
\end{aligned}
\sqcup
\begin{aligned}
\begin{tikzpicture}
 \def\R{1} 
\shadedraw[shading=ball,ball color=red](0.4*\R,0.3*\R) arc (asin(0.6):180-asin(0.6):0.5*\R) arc (180:360:0.4 and 0.06) -- cycle;
 \fill[gray, ultra nearly transparent] (0.4*\R,0.3*\R) arc (asin(0.6):-180-asin(0.6):0.5*\R) arc (-180:0:0.4 and 0.06) -- cycle;
\fill[gray, semitransparent] (0,0.3*\R) ellipse (0.4 and 0.05); 
\end{tikzpicture}
\end{aligned}\\
Mor(U_{S^{2}}) & & :=
\begin{aligned}
\begin{tikzpicture}
 \def\R{1} 
\shadedraw[shading=ball,ball color=red](0.4*\R,-0.3*\R) arc (-asin(0.6):-180+asin(0.6):0.5*\R) arc (180:360:0.4 and 0.06) -- cycle;
\filldraw[gray, ultra nearly transparent] (0.4*\R,-0.3*\R) arc (-asin(0.6):180+asin(0.6):0.5*\R) arc (180:0:0.4 and 0.06)-- cycle;

\fill[gray, semitransparent] (0,-0.3*\R) ellipse (0.4 and 0.05); 
\filldraw[fill=white, draw=black, fill opacity=0] (0.4*\R,-0.3*\R) circle (0.06);
\filldraw[fill=white, draw=black, fill opacity=0] (0,-0.5*\R) circle (0.06);
\draw[thick, double] (0.4*\R,-0.3*\R) arc (-asin(0.6):-90:0.5*\R);
\end{tikzpicture} 
\end{aligned}
\sqcup
\begin{aligned}
\begin{tikzpicture}
\def\R{1} 
\shadedraw[shading=ball,ball color=red](0.4*\R,-0.3*\R) arc (-asin(0.6):-180+asin(0.6):0.5*\R) arc (180:360:0.4 and 0.06) -- cycle;
\filldraw[gray, ultra nearly transparent] (0.4*\R,-0.3*\R) arc (-asin(0.6):180+asin(0.6):0.5*\R) arc (180:0:0.4 and 0.06)-- cycle;

\fill[gray, semitransparent] (0,-0.3*\R) ellipse (0.4 and 0.05); 
\filldraw[fill=white, draw=black, fill opacity=0] (-0.4*\R,-0.3*\R) circle (0.06);
\filldraw[fill=white, draw=black, fill opacity=0] (0,-0.5*\R) circle (0.06);
\draw[thick, double] (0,-0.5*\R) arc (-90:-180+asin(0.6):0.5*\R);
\end{tikzpicture} 
\end{aligned}
\sqcup
\begin{aligned}   
\begin{tikzpicture}
\def\R{1} 
\shadedraw[shading=ball,ball color=red, white] (0,0) circle (0.5*\R);
\draw[thick, double] (0,0) circle (0.5*\R);
\end{tikzpicture} 
\end{aligned}
\sqcup
\begin{aligned}  
\begin{tikzpicture}
 \def\R{1} 
\shadedraw[shading=ball,ball color=red](0.4*\R,0.3*\R) arc (asin(0.6):180-asin(0.6):0.5*\R) arc (180:360:0.4 and 0.06) -- cycle;
 \fill[gray, ultra nearly transparent] (0.4*\R,0.3*\R) arc (asin(0.6):-180-asin(0.6):0.5*\R) arc (-180:0:0.4 and 0.06) -- cycle;
\fill[gray, semitransparent] (0,0.3*\R) ellipse (0.4 and 0.05); 
\filldraw[fill=white, draw=black, fill opacity=0] (0.4*\R,0.3*\R) circle (0.06);
\filldraw[fill=white, draw=black, fill opacity=0] (0,0.5*\R) circle (0.06);
\draw[thick, double] (0,0.5*\R) arc (90:asin(0.6):0.5*\R);
\end{tikzpicture}
\end{aligned}
\sqcup
\begin{aligned}
\begin{tikzpicture}
 \def\R{1} 
\shadedraw[shading=ball,ball color=red](0.4*\R,0.3*\R) arc (asin(0.6):180-asin(0.6):0.5*\R) arc (180:360:0.4 and 0.06) -- cycle;
 \fill[gray, ultra nearly transparent] (0.4*\R,0.3*\R) arc (asin(0.6):-180-asin(0.6):0.5*\R) arc (-180:0:0.4 and 0.06) -- cycle;
\fill[gray, semitransparent] (0,0.3*\R) ellipse (0.4 and 0.05); 
\filldraw[fill=white, draw=black, fill opacity=0] (-0.4*\R,0.3*\R) circle (0.06);
\filldraw[fill=white, draw=black, fill opacity=0] (0,0.5*\R) circle (0.06);
\draw[thick, double] (-0.4*\R,0.3*\R) arc (180-asin(0.6):90:0.5*\R);
\end{tikzpicture}
\end{aligned}
\sqcup\\
&& \hspace{1cm} 
\begin{aligned}
\begin{tikzpicture}
 \def\R{1} 
\shadedraw[shading=ball,ball color=red](0.4*\R,-0.3*\R) arc (-asin(0.6):-180+asin(0.6):0.5*\R) arc (180:360:0.4 and 0.06) -- cycle;
 \filldraw[gray, ultra nearly transparent] (0.4*\R,-0.3*\R) arc (-asin(0.6):180+asin(0.6):0.5*\R) arc (180:0:0.4 and 0.06)-- cycle;

\fill[gray, semitransparent] (0,-0.3*\R) ellipse (0.4 and 0.05); 
 \filldraw[fill=white, draw=black, fill opacity=0] (0.4*\R,-0.3*\R) circle (0.06);
 \filldraw[fill=white, draw=black, fill opacity=0] (0,-0.5*\R) circle (0.06);
 \draw[thick, double] (0.4*\R,-0.3*\R) arc (-asin(0.6):-90:0.5*\R);
\end{tikzpicture} 
\end{aligned}
\sqcup
\begin{aligned}
\begin{tikzpicture}
\def\R{1} 
\shadedraw[shading=ball,ball color=red](0.4*\R,-0.3*\R) arc (-asin(0.6):-180+asin(0.6):0.5*\R) arc (180:360:0.4 and 0.06) -- cycle;
\filldraw[gray, ultra nearly transparent] (0.4*\R,-0.3*\R) arc (-asin(0.6):180+asin(0.6):0.5*\R) arc (180:0:0.4 and 0.06)-- cycle;

\fill[gray, semitransparent] (0,-0.3*\R) ellipse (0.4 and 0.05); 
\filldraw[fill=white, draw=black, fill opacity=0] (-0.4*\R,-0.3*\R) circle (0.06);
\filldraw[fill=white, draw=black, fill opacity=0] (0,-0.5*\R) circle (0.06);
\draw[thick, double] (0,-0.5*\R) arc (-90:-180+asin(0.6):0.5*\R);
\end{tikzpicture} 
\end{aligned}
\sqcup
\begin{aligned}
\begin{tikzpicture}
\def\R{1} 
\shadedraw[shading=ball,ball color=red, white] (0,0) circle (0.5*\R);
\draw[thick, double] (0,0) circle (0.5*\R);
\end{tikzpicture} 
\end{aligned}
\sqcup
\begin{aligned}  
\begin{tikzpicture}
 \def\R{1} 
\shadedraw[shading=ball,ball color=red](0.4*\R,0.3*\R) arc (asin(0.6):180-asin(0.6):0.5*\R) arc (180:360:0.4 and 0.06) -- cycle;
 \fill[gray, ultra nearly transparent] (0.4*\R,0.3*\R) arc (asin(0.6):-180-asin(0.6):0.5*\R) arc (-180:0:0.4 and 0.06) -- cycle;
\fill[gray, semitransparent] (0,0.3*\R) ellipse (0.4 and 0.05); 
\filldraw[fill=white, draw=black, fill opacity=0] (0.4*\R,0.3*\R) circle (0.06);
\filldraw[fill=white, draw=black, fill opacity=0] (0,0.5*\R) circle (0.06);
\draw[thick, double] (0,0.5*\R) arc (90:asin(0.6):0.5*\R);
\end{tikzpicture}
\end{aligned}
\sqcup
\begin{aligned}
\begin{tikzpicture}
 \def\R{1} 
\shadedraw[shading=ball,ball color=red](0.4*\R,0.3*\R) arc (asin(0.6):180-asin(0.6):0.5*\R) arc (180:360:0.4 and 0.06) -- cycle;
 \fill[gray, ultra nearly transparent] (0.4*\R,0.3*\R) arc (asin(0.6):-180-asin(0.6):0.5*\R) arc (-180:0:0.4 and 0.06) -- cycle;
\fill[gray, semitransparent] (0,0.3*\R) ellipse (0.4 and 0.05); 
\filldraw[fill=white, draw=black, fill opacity=0] (-0.4*\R,0.3*\R) circle (0.06);
\filldraw[fill=white, draw=black, fill opacity=0] (0,0.5*\R) circle (0.06);
\draw[thick, double] (-0.4*\R,0.3*\R) arc (180-asin(0.6):90:0.5*\R);
\end{tikzpicture}
\end{aligned}\\
&&\hspace{1cm} \sqcup
\begin{aligned}
\begin{tikzpicture}
 \def\R{1} 
\shadedraw[shading=ball,ball color=red](0.4*\R,-0.3*\R) arc (-asin(0.6):-180+asin(0.6):0.5*\R) arc (180:360:0.4 and 0.06) -- cycle;
 \filldraw[gray, ultra nearly transparent] (0.4*\R,-0.3*\R) arc (-asin(0.6):180+asin(0.6):0.5*\R) arc (180:0:0.4 and 0.06)-- cycle;

\fill[gray, semitransparent] (0,-0.3*\R) ellipse (0.4 and 0.05); 
\end{tikzpicture} 
\end{aligned}
 \sqcup
 \begin{aligned}
\begin{tikzpicture}
 \def\R{1} 
 \shadedraw[shading=ball,ball color=red, white] (0,0) circle (0.5*\R);
 \filldraw[fill=white, draw=black, fill opacity=0] (0,0.5*\R) circle (0.06);
 \filldraw[fill=white, draw=black, fill opacity=0] (0,-0.5*\R) circle (0.06);
\draw[thick, double] (0,0.5*\R) arc (90:-90:0.5*\R);
\end{tikzpicture} 
\end{aligned}
 \sqcup
 \begin{aligned}   
\begin{tikzpicture}
 \def\R{1} 
 \shadedraw[shading=ball,ball color=red, white] (0,0) circle (0.5*\R);
  \filldraw[fill=white, draw=black, fill opacity=0] (0,0.5*\R) circle (0.06);
 \filldraw[fill=white, draw=black, fill opacity=0] (0,-0.5*\R) circle (0.06);
 \draw[thick, double] (0,-0.5*\R) arc (-90:-270:0.5*\R);
\end{tikzpicture} 
\end{aligned}
 \sqcup
 \begin{aligned}
\begin{tikzpicture}
 \def\R{1} 
\shadedraw[shading=ball,ball color=red](0.4*\R,0.3*\R) arc (asin(0.6):180-asin(0.6):0.5*\R) arc (180:360:0.4 and 0.06) -- cycle;
 \fill[gray, ultra nearly transparent] (0.4*\R,0.3*\R) arc (asin(0.6):-180-asin(0.6):0.5*\R) arc (-180:0:0.4 and 0.06) -- cycle;
\fill[gray, semitransparent] (0,0.3*\R) ellipse (0.4 and 0.05); 
\end{tikzpicture}
\end{aligned}
\end{eqnarray*}

Define a Lie groupoid $\gG_{\E_{v,C}}$ by

\begin{eqnarray*}
Ob(\gG_{\E_{v,C}}):=
\begin{aligned}
\begin{tikzpicture}
 \def\R{1} 
\shadedraw[shading=ball,ball color=red](0.4*\R,-0.3*\R) arc (-asin(0.6):-180+asin(0.6):0.5*\R) arc (180:360:0.4 and 0.06) -- cycle;
 \filldraw[gray, ultra nearly transparent] (0.4*\R,-0.3*\R) arc (-asin(0.6):180+asin(0.6):0.5*\R) arc (180:0:0.4 and 0.06)-- cycle;
\fill[gray, semitransparent] (0,-0.3*\R) ellipse (0.4 and 0.05); 
\end{tikzpicture} 
\end{aligned}
\times V_{C}\negthickspace && \sqcup (
\begin{aligned}
\begin{tikzpicture}
 \def\R{1} 
\shadedraw[shading=ball,ball color=red](0.5*\R,0) arc (0:-180:0.5*\R) arc (180:360:0.5 and 0.1) -- cycle;
 \fill[gray, ultra nearly transparent] (0.5*\R,0) arc (0:180:0.5*\R) arc (180:0:0.5 and 0.1) -- cycle;
\filldraw[fill=gray, draw=black, semitransparent] (0,0) ellipse (0.5 and 0.1); 
 \filldraw[fill=white, draw=black, fill opacity=0] (0.5*\R,0) circle (0.06);
 \filldraw[fill=white, draw=black, fill opacity=0] (0,-0.5*\R) circle (0.06);
\draw[thick, double] (0.5*\R,0) arc (0:-90:0.5*\R);
\end{tikzpicture} 
\end{aligned}
\times V_{C}\ \sqcup
\begin{aligned}
\begin{tikzpicture}
 \def\R{1} 
\shadedraw[shading=ball,ball color=red](0.5*\R,0) arc (0:180:0.5*\R) arc (180:360:0.5 and 0.1) -- cycle;
 \fill[gray, ultra   nearly transparent] (0.5*\R,0) arc (0:-180:0.5*\R) arc (-180:0:0.5 and 0.1)-- cycle;
\filldraw[fill=gray, draw=gray, semitransparent] (0,0) ellipse (0.5 and 0.1); 
\draw[densely dashed] (0.5*\R,0) arc (0:180:0.5 and 0.1);
\draw (-0.5*\R,0) arc (180:360:0.5 and 0.1);
  \filldraw[fill=white, draw=black, fill opacity=0] (0,0.5*\R) circle (0.06);
 \filldraw[fill=white, draw=black, fill opacity=0] (0.5*\R,0) circle (0.06);
 \draw[thick, double] (0.5*\R,0) arc (0:90:0.5*\R);
\end{tikzpicture} 
\end{aligned}
\times V_{C})/rel_{Ob,1}\ \sqcup \\
\begin{aligned}
\begin{tikzpicture}
 \def\R{1} 
\shadedraw[shading=ball,ball color=red](0.4*\R,0.3*\R) arc (asin(0.6):180-asin(0.6):0.5*\R) arc (180:360:0.4 and 0.06) -- cycle;
 \fill[gray, ultra nearly transparent] (0.4*\R,0.3*\R) arc (asin(0.6):-180-asin(0.6):0.5*\R) arc (-180:0:0.4 and 0.06) -- cycle;
\fill[gray, semitransparent] (0,0.3*\R) ellipse (0.4 and 0.05); 
\end{tikzpicture}
\end{aligned}
\times V_{C} \negthickspace\negthickspace\negthickspace & & \sqcup  
(   
\begin{aligned}
\begin{tikzpicture}
\def\R{1} 
\shadedraw[shading=ball,ball color=red](0.5*\R,0) arc (0:-180:0.5*\R) arc (180:360:0.5 and 0.1) -- cycle;
 \fill[gray, ultra nearly transparent] (0.5*\R,0) arc (0:180:0.5*\R) arc (180:0:0.5 and 0.1) -- cycle;
\filldraw[fill=gray, draw=black, semitransparent] (0,0) ellipse (0.5 and 0.1); 
\filldraw[fill=white, draw=black, fill opacity=0] (-0.5*\R,0) circle (0.06);
\filldraw[fill=white, draw=black, fill opacity=0] (0,-0.5*\R) circle (0.06);
\draw[thick, double] (0,-0.5*\R) arc (-90:-180:0.5*\R);
\end{tikzpicture} 
\end{aligned}
\times V_{C}\ \sqcup 
\begin{aligned}
\begin{tikzpicture}
 \def\R{1} 
\shadedraw[shading=ball,ball color=red](0.5*\R,0) arc (0:180:0.5*\R) arc (180:360:0.5 and 0.1) -- cycle;
 \fill[gray, ultra   nearly transparent] (0.5*\R,0) arc (0:-180:0.5*\R) arc (-180:0:0.5 and 0.1)-- cycle;
\filldraw[fill=gray, draw=gray, semitransparent] (0,0) ellipse (0.5 and 0.1); 
\draw[densely dashed] (0.5*\R,0) arc (0:180:0.5 and 0.1);
\draw (-0.5*\R,0) arc (180:360:0.5 and 0.1);
\filldraw[fill=white, draw=black, fill opacity=0] (0,0.5*\R) circle (0.06);
\filldraw[fill=white, draw=black, fill opacity=0] (-0.5*\R,0) circle (0.06);
\draw[thick, double] (-0.5*\R,0) arc (-180:-270:0.5*\R);
\end{tikzpicture} 
\end{aligned}
\times V_{C} )/ rel_{Ob,2}
\end{eqnarray*}
\begin{eqnarray*}
Mor(\gG_{\E_{v,C}}):=
\begin{aligned}
\begin{tikzpicture}
 \def\R{1} 
\shadedraw[shading=ball,ball color=red](0.4*\R,-0.3*\R) arc (-asin(0.6):-180+asin(0.6):0.5*\R) arc (180:360:0.4 and 0.06) -- cycle;
 \filldraw[gray, ultra nearly transparent] (0.4*\R,-0.3*\R) arc (-asin(0.6):180+asin(0.6):0.5*\R) arc (180:0:0.4 and 0.06)-- cycle;

\fill[gray, semitransparent] (0,-0.3*\R) ellipse (0.4 and 0.05); 
 \filldraw[fill=white, draw=black, fill opacity=0] (0.4*\R,-0.3*\R) circle (0.06);
 \filldraw[fill=white, draw=black, fill opacity=0] (0,-0.5*\R) circle (0.06);
 \draw[thick, double] (0.4*\R,-0.3*\R) arc (-asin(0.6):-90:0.5*\R);
\end{tikzpicture} 
\end{aligned}
\times G_{C}\times V_{C}\ \sqcup
\begin{aligned}
\begin{tikzpicture}
\def\R{1} 
\shadedraw[shading=ball,ball color=red](0.4*\R,-0.3*\R) arc (-asin(0.6):-180+asin(0.6):0.5*\R) arc (180:360:0.4 and 0.06) -- cycle;
\filldraw[gray, ultra nearly transparent] (0.4*\R,-0.3*\R) arc (-asin(0.6):180+asin(0.6):0.5*\R) arc (180:0:0.4 and 0.06)-- cycle;

\fill[gray, semitransparent] (0,-0.3*\R) ellipse (0.4 and 0.05); 
\filldraw[fill=white, draw=black, fill opacity=0] (-0.4*\R,-0.3*\R) circle (0.06);
\filldraw[fill=white, draw=black, fill opacity=0] (0,-0.5*\R) circle (0.06);
\draw[thick, double] (0,-0.5*\R) arc (-90:-180+asin(0.6):0.5*\R);
\end{tikzpicture} 
\end{aligned}
\times G_{C}\times V_{C} \negthickspace\negthickspace\negthickspace & & \sqcup \\
(
\begin{aligned}
\begin{tikzpicture}
 \def\R{1} 
\shadedraw[shading=ball,ball color=red](0.5*\R,0) arc (0:-180:0.5*\R) arc (180:360:0.5 and 0.1) -- cycle;
  \fill[gray, ultra nearly transparent] (0.5*\R,0) arc (0:180:0.5*\R) arc (180:0:0.5 and 0.1) -- cycle;
\filldraw[fill=gray, draw=black,  semitransparent] (0,0) ellipse (0.5 and 0.1); 
 \filldraw[fill=white, draw=black, fill opacity=0] (0.5*\R,0) circle (0.06);
 \filldraw[fill=white, draw=black, fill opacity=0] (-0.5*\R,0) circle (0.06);
\draw[thick, double] (0.5*\R,0) arc (0:-180:0.5*\R);
\end{tikzpicture} 
\end{aligned}
\times G_{C}\times V_{C} \negthickspace\negthickspace\negthickspace & & \sqcup
\begin{aligned}
\begin{tikzpicture}
 \def\R{1} 
\shadedraw[shading=ball,ball color=red](0.5*\R,0) arc (0:180:0.5*\R) arc (180:360:0.5 and 0.1) -- cycle;
 \fill[gray, ultra   nearly transparent] (0.5*\R,0) arc (0:-180:0.5*\R) arc (-180:0:0.5 and 0.1)-- cycle;
\filldraw[fill=gray, draw=gray, semitransparent] (0,0) ellipse (0.5 and 0.1); 
\draw[densely dashed] (0.5*\R,0) arc (0:180:0.5 and 0.1);
\draw (-0.5*\R,0) arc (180:360:0.5 and 0.1);  
\filldraw[fill=white, draw=black, fill opacity=0] (-0.5*\R,0) circle (0.06);
 \filldraw[fill=white, draw=black, fill opacity=0] (0.5*\R,0) circle (0.06);
 \draw[thick, double] (0.5*\R,0) arc (0:180:0.5*\R);
\end{tikzpicture} 
\end{aligned}
\times G_{C}\times V_{C})/rel_{Mor,1}\\
\negthickspace\negthickspace\negthickspace & & \sqcup
\begin{aligned}
\begin{tikzpicture}
 \def\R{1} 
\shadedraw[shading=ball,ball color=red](0.4*\R,0.3*\R) arc (asin(0.6):180-asin(0.6):0.5*\R) arc (180:360:0.4 and 0.06) -- cycle;
 \fill[gray, ultra nearly transparent] (0.4*\R,0.3*\R) arc (asin(0.6):-180-asin(0.6):0.5*\R) arc (-180:0:0.4 and 0.06) -- cycle;
\fill[gray, semitransparent] (0,0.3*\R) ellipse (0.4 and 0.05); 
\filldraw[fill=white, draw=black, fill opacity=0] (0.4*\R,0.3*\R) circle (0.06);
\filldraw[fill=white, draw=black, fill opacity=0] (0,0.5*\R) circle (0.06);
\draw[thick, double] (0,0.5*\R) arc (90:asin(0.6):0.5*\R);
\end{tikzpicture}
\end{aligned}
\times G_{C}\times V_{C}\ \sqcup
\begin{aligned}
\begin{tikzpicture}
 \def\R{1} 
\shadedraw[shading=ball,ball color=red](0.4*\R,0.3*\R) arc (asin(0.6):180-asin(0.6):0.5*\R) arc (180:360:0.4 and 0.06) -- cycle;
 \fill[gray, ultra nearly transparent] (0.4*\R,0.3*\R) arc (asin(0.6):-180-asin(0.6):0.5*\R) arc (-180:0:0.4 and 0.06) -- cycle;
\fill[gray, semitransparent] (0,0.3*\R) ellipse (0.4 and 0.05); 
\filldraw[fill=white, draw=black, fill opacity=0] (-0.4*\R,0.3*\R) circle (0.06);
\filldraw[fill=white, draw=black, fill opacity=0] (0,0.5*\R) circle (0.06);
\draw[thick, double] (-0.4*\R,0.3*\R) arc (180-asin(0.6):90:0.5*\R);
\end{tikzpicture}
\end{aligned}
\times G_{C}\times V_{C}\\
\sqcup
\begin{aligned}
\begin{tikzpicture}
 \def\R{1} 
\shadedraw[shading=ball,ball color=red](0.4*\R,-0.3*\R) arc (-asin(0.6):-180+asin(0.6):0.5*\R) arc (180:360:0.4 and 0.06) -- cycle;
 \filldraw[gray, ultra nearly transparent] (0.4*\R,-0.3*\R) arc (-asin(0.6):180+asin(0.6):0.5*\R) arc (180:0:0.4 and 0.06)-- cycle;

\fill[gray, semitransparent] (0,-0.3*\R) ellipse (0.4 and 0.05); 
 \filldraw[fill=white, draw=black, fill opacity=0] (0.4*\R,-0.3*\R) circle (0.06);
 \filldraw[fill=white, draw=black, fill opacity=0] (0,-0.5*\R) circle (0.06);
 \draw[thick, double] (0.4*\R,-0.3*\R) arc (-asin(0.6):-90:0.5*\R);
\end{tikzpicture} 
\end{aligned}
\times G_{C}\times V_{C}\ \sqcup
\begin{aligned}
\begin{tikzpicture}
\def\R{1} 
\shadedraw[shading=ball,ball color=red](0.4*\R,-0.3*\R) arc (-asin(0.6):-180+asin(0.6):0.5*\R) arc (180:360:0.4 and 0.06) -- cycle;
\filldraw[gray, ultra nearly transparent] (0.4*\R,-0.3*\R) arc (-asin(0.6):180+asin(0.6):0.5*\R) arc (180:0:0.4 and 0.06)-- cycle;

\fill[gray, semitransparent] (0,-0.3*\R) ellipse (0.4 and 0.05); 
\filldraw[fill=white, draw=black, fill opacity=0] (-0.4*\R,-0.3*\R) circle (0.06);
\filldraw[fill=white, draw=black, fill opacity=0] (0,-0.5*\R) circle (0.06);
\draw[thick, double] (0,-0.5*\R) arc (-90:-180+asin(0.6):0.5*\R);
\end{tikzpicture} 
\end{aligned}
\times G_{C}\times V_{C}\negthickspace\negthickspace\negthickspace & &\sqcup \\
(
\begin{aligned}
\begin{tikzpicture}
 \def\R{1} 
\shadedraw[shading=ball,ball color=red](0.5*\R,0) arc (0:-180:0.5*\R) arc (180:360:0.5 and 0.1) -- cycle;
  \fill[gray, ultra nearly transparent] (0.5*\R,0) arc (0:180:0.5*\R) arc (180:0:0.5 and 0.1) -- cycle;
\filldraw[fill=gray, draw=black,  semitransparent] (0,0) ellipse (0.5 and 0.1); 
 \filldraw[fill=white, draw=black, fill opacity=0] (0.5*\R,0) circle (0.06);
 \filldraw[fill=white, draw=black, fill opacity=0] (-0.5*\R,0) circle (0.06);
\draw[thick, double] (0.5*\R,0) arc (0:-180:0.5*\R);
\end{tikzpicture} 
\end{aligned}
\times G_{C}\times V_{C} \negthickspace\negthickspace\negthickspace & & \sqcup
\begin{aligned}
\begin{tikzpicture}
 \def\R{1} 
\shadedraw[shading=ball,ball color=red](0.5*\R,0) arc (0:180:0.5*\R) arc (180:360:0.5 and 0.1) -- cycle;
 \fill[gray, ultra   nearly transparent] (0.5*\R,0) arc (0:-180:0.5*\R) arc (-180:0:0.5 and 0.1)-- cycle;
\filldraw[fill=gray, draw=gray, semitransparent] (0,0) ellipse (0.5 and 0.1); 
\draw[densely dashed] (0.5*\R,0) arc (0:180:0.5 and 0.1);
\draw (-0.5*\R,0) arc (180:360:0.5 and 0.1);  
\filldraw[fill=white, draw=black, fill opacity=0] (-0.5*\R,0) circle (0.06);
 \filldraw[fill=white, draw=black, fill opacity=0] (0.5*\R,0) circle (0.06);
 \draw[thick, double] (0.5*\R,0) arc (0:180:0.5*\R);
\end{tikzpicture} 
\end{aligned}
\times G_{C}\times V_{C})/rel_{Mor,2}\\
\negthickspace\negthickspace\negthickspace & & \sqcup
\begin{aligned}
\begin{tikzpicture}
 \def\R{1} 
\shadedraw[shading=ball,ball color=red](0.4*\R,0.3*\R) arc (asin(0.6):180-asin(0.6):0.5*\R) arc (180:360:0.4 and 0.06) -- cycle;
 \fill[gray, ultra nearly transparent] (0.4*\R,0.3*\R) arc (asin(0.6):-180-asin(0.6):0.5*\R) arc (-180:0:0.4 and 0.06) -- cycle;
\fill[gray, semitransparent] (0,0.3*\R) ellipse (0.4 and 0.05); 
\filldraw[fill=white, draw=black, fill opacity=0] (0.4*\R,0.3*\R) circle (0.06);
\filldraw[fill=white, draw=black, fill opacity=0] (0,0.5*\R) circle (0.06);
\draw[thick, double] (0,0.5*\R) arc (90:asin(0.6):0.5*\R);
\end{tikzpicture}
\end{aligned}
\times G_{C}\times V_{C}\  \sqcup
\begin{aligned}
\begin{tikzpicture}
 \def\R{1} 
\shadedraw[shading=ball,ball color=red](0.4*\R,0.3*\R) arc (asin(0.6):180-asin(0.6):0.5*\R) arc (180:360:0.4 and 0.06) -- cycle;
 \fill[gray, ultra nearly transparent] (0.4*\R,0.3*\R) arc (asin(0.6):-180-asin(0.6):0.5*\R) arc (-180:0:0.4 and 0.06) -- cycle;
\fill[gray, semitransparent] (0,0.3*\R) ellipse (0.4 and 0.05); 
\filldraw[fill=white, draw=black, fill opacity=0] (-0.4*\R,0.3*\R) circle (0.06);
\filldraw[fill=white, draw=black, fill opacity=0] (0,0.5*\R) circle (0.06);
\draw[thick, double] (-0.4*\R,0.3*\R) arc (180-asin(0.6):90:0.5*\R);
\end{tikzpicture}
\end{aligned}
\times G_{C}\times V_{C}\\
\sqcup
\begin{aligned}
\begin{tikzpicture}
 \def\R{1} 
\shadedraw[shading=ball,ball color=red](0.4*\R,-0.3*\R) arc (-asin(0.6):-180+asin(0.6):0.5*\R) arc (180:360:0.4 and 0.06) -- cycle;
 \filldraw[gray, ultra nearly transparent] (0.4*\R,-0.3*\R) arc (-asin(0.6):180+asin(0.6):0.5*\R) arc (180:0:0.4 and 0.06)-- cycle;

\fill[gray, semitransparent] (0,-0.3*\R) ellipse (0.4 and 0.05); 
\end{tikzpicture} 
\end{aligned}
\times G_{C}\times V_{C}\  \sqcup (
\begin{aligned}
\begin{tikzpicture}
 \def\R{1} 
\shadedraw[shading=ball,ball color=red](0.5*\R,0) arc (0:-180:0.5*\R) arc (180:360:0.5 and 0.1) -- cycle;
  \fill[gray, ultra nearly transparent] (0.5*\R,0) arc (0:180:0.5*\R) arc (180:0:0.5 and 0.1) -- cycle;
\filldraw[fill=gray, draw=black, semitransparent] (0,0) ellipse (0.5 and 0.1); 
 \filldraw[fill=white, draw=black, fill opacity=0] (0.5*\R,0) circle (0.06);
 \filldraw[fill=white, draw=black, fill opacity=0] (0,-0.5*\R) circle (0.06);
\draw[thick, double] (0.5*\R,0) arc (0:-90:0.5*\R);
\end{tikzpicture} 
\end{aligned}
\times G_{C}\times V_{C}  \negthickspace\negthickspace\negthickspace & & \sqcup
\begin{aligned}
\begin{tikzpicture}
 \def\R{1} 
\shadedraw[shading=ball,ball color=red](0.5*\R,0) arc (0:180:0.5*\R) arc (180:360:0.5 and 0.1) -- cycle;
 \fill[gray, ultra   nearly transparent] (0.5*\R,0) arc (0:-180:0.5*\R) arc (-180:0:0.5 and 0.1)-- cycle;
\filldraw[fill=gray, draw=gray, semitransparent] (0,0) ellipse (0.5 and 0.1); 
\draw[densely dashed] (0.5*\R,0) arc (0:180:0.5 and 0.1);
\draw (-0.5*\R,0) arc (180:360:0.5 and 0.1);  
\filldraw[fill=white, draw=black, fill opacity=0] (0,0.5*\R) circle (0.06);
 \filldraw[fill=white, draw=black, fill opacity=0] (0.5*\R,0) circle (0.06);
 \draw[thick, double] (0.5*\R,0) arc (0:90:0.5*\R);
\end{tikzpicture} 
\end{aligned}
\times G_{C}\times V_{C})/rel_{Mor,3}\\
\ \sqcup 
\begin{aligned}
\begin{tikzpicture}
 \def\R{1} 
\shadedraw[shading=ball,ball color=red](0.4*\R,0.3*\R) arc (asin(0.6):180-asin(0.6):0.5*\R) arc (180:360:0.4 and 0.06) -- cycle;
 \fill[gray, ultra nearly transparent] (0.4*\R,0.3*\R) arc (asin(0.6):-180-asin(0.6):0.5*\R) arc (-180:0:0.4 and 0.06) -- cycle;
\fill[gray, semitransparent] (0,0.3*\R) ellipse (0.4 and 0.05); 
\end{tikzpicture}
\end{aligned}
 \times G_{C}\times V_{C}\sqcup 
(   
\begin{aligned}
\begin{tikzpicture}
\def\R{1} 
\shadedraw[shading=ball,ball color=red](0.5*\R,0) arc (0:-180:0.5*\R) arc (180:360:0.5 and 0.1) -- cycle;
 \fill[gray, ultra nearly transparent] (0.5*\R,0) arc (0:180:0.5*\R) arc (180:0:0.5 and 0.1) -- cycle;
\filldraw[fill=gray, draw=black,  semitransparent] (0,0) ellipse (0.5 and 0.1); 
\filldraw[fill=white, draw=black, fill opacity=0] (-0.5*\R,0) circle (0.06);
\filldraw[fill=white, draw=black, fill opacity=0] (0,-0.5*\R) circle (0.06);
\draw[thick, double] (0,-0.5*\R) arc (-90:-180:0.5*\R);
\end{tikzpicture} 
\end{aligned}
\times  G_{C}\times V_{C} \negthickspace\negthickspace\negthickspace & & \sqcup 
\begin{aligned}
\begin{tikzpicture}
 \def\R{1} 
\shadedraw[shading=ball,ball color=red](0.5*\R,0) arc (0:180:0.5*\R) arc (180:360:0.5 and 0.1) -- cycle;
 \fill[gray, ultra   nearly transparent] (0.5*\R,0) arc (0:-180:0.5*\R) arc (-180:0:0.5 and 0.1)-- cycle;
\filldraw[fill=gray, draw=gray, semitransparent] (0,0) ellipse (0.5 and 0.1); 
\draw[densely dashed] (0.5*\R,0) arc (0:180:0.5 and 0.1);
\draw (-0.5*\R,0) arc (180:360:0.5 and 0.1);
\filldraw[fill=white, draw=black, fill opacity=0] (0,0.5*\R) circle (0.06);
\filldraw[fill=white, draw=black, fill opacity=0] (-0.5*\R,0) circle (0.06);
\draw[thick, double] (-0.5*\R,0) arc (-180:-270:0.5*\R);
\end{tikzpicture} 
\end{aligned}
\times G_{C}\times V_{C})/ rel_{Mor,4}.
\end{eqnarray*}

For notational convenience, we identify hemispheres with unit disks.

The glueing along the boundary is given by $\tilde{\gamma}_{v}:U_{S^{1}}\times G_{C}\ltimes V_{C}\to G_{C}\ltimes V_{C}$. More explicitly, 
\begin{itemize}
\item $rel_{Ob,1}$: for $(e^{i2\pi t},\vec{w})$ in the boundary of 
$
\begin{aligned}
\begin{tikzpicture}
 \def\R{1} 
\shadedraw[shading=ball,ball color=red](0.5*\R,0) arc (0:-180:0.5*\R) arc (180:360:0.5 and 0.1) -- cycle;
  \fill[gray, ultra nearly transparent] (0.5*\R,0) arc (0:180:0.5*\R) arc (180:0:0.5 and 0.1) -- cycle;
\filldraw[fill=gray, draw=black,  semitransparent] (0,0) ellipse (0.5 and 0.1); 
 \filldraw[fill=white, draw=black, fill opacity=0] (0.5*\R,0) circle (0.06);
 \filldraw[fill=white, draw=black, fill opacity=0] (0,-0.5*\R) circle (0.06);
\draw[thick, double] (0.5*\R,0) arc (0:-90:0.5*\R);
\end{tikzpicture} 
\end{aligned}
\times V_{C}$
, $(e^{i2\pi t'},\vec{w}')$ in the boundary of 
$
\begin{aligned}
\begin{tikzpicture}
 \def\R{1} 
\shadedraw[shading=ball,ball color=red](0.5*\R,0) arc (0:180:0.5*\R) arc (180:360:0.5 and 0.1) -- cycle;
 \fill[gray, ultra   nearly transparent] (0.5*\R,0) arc (0:-180:0.5*\R) arc (-180:0:0.5 and 0.1)-- cycle;
\filldraw[fill=gray, draw=gray, semitransparent] (0,0) ellipse (0.5 and 0.1); 
\draw[densely dashed] (0.5*\R,0) arc (0:180:0.5 and 0.1);
\draw (-0.5*\R,0) arc (180:360:0.5 and 0.1);  
\filldraw[fill=white, draw=black, fill opacity=0] (0,0.5*\R) circle (0.06);
 \filldraw[fill=white, draw=black, fill opacity=0] (0.5*\R,0) circle (0.06);
 \draw[thick, double] (0.5*\R,0) arc (0:90:0.5*\R);
\end{tikzpicture} 
\end{aligned}
\times V_{C}$, $(e^{i2\pi t},\vec{w})\sim (e^{i2\pi t'},\vec{w}')$ if and only if $t'=-t$, $\vec{w}'=D(-tv)\vec{w}$. Note that here $t\in(0,1)$ and $t'\in (-1,0)$.

\item $rel_{Ob,2}$: for $(e^{i2\pi t},\vec{w})$ in the boundary of 
$
\begin{aligned}
\begin{tikzpicture}
\def\R{1} 
\shadedraw[shading=ball,ball color=red](0.5*\R,0) arc (0:-180:0.5*\R) arc (180:360:0.5 and 0.1) -- cycle;
 \fill[gray, ultra nearly transparent] (0.5*\R,0) arc (0:180:0.5*\R) arc (180:0:0.5 and 0.1) -- cycle;
\filldraw[fill=gray, draw=black,  semitransparent] (0,0) ellipse (0.5 and 0.1); 
\filldraw[fill=white, draw=black, fill opacity=0] (-0.5*\R,0) circle (0.06);
\filldraw[fill=white, draw=black, fill opacity=0] (0,-0.5*\R) circle (0.06);
\draw[thick, double] (0,-0.5*\R) arc (-90:-180:0.5*\R);
\end{tikzpicture} 
\end{aligned}
\times V_{C}$, $(e^{i2\pi t'},\vec{w}')$ in the boundary of 
$
\begin{aligned}
\begin{tikzpicture}
 \def\R{1} 
\shadedraw[shading=ball,ball color=red](0.5*\R,0) arc (0:180:0.5*\R) arc (180:360:0.5 and 0.1) -- cycle;
 \fill[gray, ultra   nearly transparent] (0.5*\R,0) arc (0:-180:0.5*\R) arc (-180:0:0.5 and 0.1)-- cycle;
\filldraw[fill=gray, draw=gray, semitransparent] (0,0) ellipse (0.5 and 0.1); 
\draw[densely dashed] (0.5*\R,0) arc (0:180:0.5 and 0.1);
\draw (-0.5*\R,0) arc (180:360:0.5 and 0.1);
\filldraw[fill=white, draw=black, fill opacity=0] (0,0.5*\R) circle (0.06);
\filldraw[fill=white, draw=black, fill opacity=0] (-0.5*\R,0) circle (0.06);
\draw[thick, double] (-0.5*\R,0) arc (-180:-270:0.5*\R);
\end{tikzpicture}
\end{aligned} 
\times V_{C}$,
 $(e^{i2\pi t},\vec{w})\sim (e^{i2\pi t'},\vec{w}')$ if and only if $t'=-t$, $\vec{w}'=D(-tv)\vec{w}$. Note that here $t,t'\in (-\frac{1}{2},\frac{1}{2})$.

 \item $rel_{Mor,i}, \ i=1,2,3,4$: for boundary elements $(e^{i2\pi t_{1}}\to e^{i2\pi t_{2}}, \vec{w}\xrightarrow{g} g\cdot \vec{w})$ and $(e^{i2\pi t_{1}'}\to e^{i2\pi t_{2}'}, \vec{w}'\xrightarrow{g'} g'\cdot \vec{w}')$, 
 $$(e^{i2\pi t_{1}}\to e^{i2\pi t_{2}}, \vec{w}\xrightarrow{g} g\cdot \vec{w}) \sim(e^{i2\pi t_{1}'}\to e^{i2\pi t_{2}'}, \vec{w}'\xrightarrow{g'} g'\cdot \vec{w}')$$ 
 if and only if 
\begin{eqnarray*}
t_{1}=-t_{1}' & & ,\ \ t_{2}=-t_{2}',\\
\vec{w}'=D(-t_{1}v)\vec{w} & & ,\  g'\cdot \vec{w}'=D(-t_{2}v) g\cdot \vec{w}.
\end{eqnarray*}
\end{itemize}

We denote $m_{v}:=ord(D(-v))=ord(D(v))$. Then $D(-v)$ generates a cyclic subgroup of $G_{C}$ which is isomorphic to $\inte_{m_{v}}$. We denote $\mathfrak{i}:\inte_{m_{v}}\to G_{C}$ the inclusion of the subgroup.

Note that when $v=y_{i}$, $i=1,...,N$, $m_{v}$ is the number labeling the corresponding facet of the polytope $\Delta$.
\begin{prop}\label{constantsection}
 If $\pmb{s}_{x}: (S^{2}_{orb},p)\to \E_{v}$ is a sectional orbifold morphism from an orbisphere with one orbipoint (possibly a trivial one) lifting $s_{x}$, then the orbifold structure group at $p$ is $\inte_{m_{v}}$.
 \end{prop}
\begin{proof}
Given any orbifold morphism from a sphere with at most one orbifold point at $p$, it determines an orbifold morphism $\check{\pmb{s}}$ away from $p$ by restriction. Without loss of generality, we may assume $p$ to be the north pole of the sphere. Since there is no other orbifold point on the sphere and $\pmb{s}$ lifts the zero section, $\check{\pmb{s}}$ can be represented by the groupoid morphism:
\begin{eqnarray*}
\check{\pmb{s}}_{0}=(id,0): \ Ob(U_{\check{S}^{2}})=
\begin{aligned}
\begin{tikzpicture}
 \def\R{1} 
\shadedraw[shading=ball,ball color=red](0.4*\R,-0.3*\R) arc (-asin(0.6):-180+asin(0.6):0.5*\R) arc (180:360:0.4 and 0.06) -- cycle;
 \filldraw[gray, ultra nearly transparent] (0.4*\R,-0.3*\R) arc (-asin(0.6):180+asin(0.6):0.5*\R) arc (180:0:0.4 and 0.06)-- cycle;

\fill[gray, semitransparent] (0,-0.3*\R) ellipse (0.4 and 0.05); 
\end{tikzpicture} 
\end{aligned}
 \sqcup  
 \begin{aligned}
\begin{tikzpicture}
 \def\R{1} 
 \shadedraw[shading=ball,ball color=red, white] (0,0) circle (0.5*\R);
 \filldraw[fill=white, draw=black, fill opacity=0] (0,0.5*\R) circle (0.06);
 \filldraw[fill=white, draw=black, fill opacity=0] (0,-0.5*\R) circle (0.06);
\draw[thick, double] (0,0.5*\R) arc (90:-90:0.5*\R);
\end{tikzpicture} 
\end{aligned}
\sqcup 
\begin{aligned}
\begin{tikzpicture}
 \def\R{1} 
 \shadedraw[shading=ball,ball color=red, white] (0,0) circle (0.5*\R);
  \filldraw[fill=white, draw=black, fill opacity=0] (0,0.5*\R) circle (0.06);
 \filldraw[fill=white, draw=black, fill opacity=0] (0,-0.5*\R) circle (0.06);
 \draw[thick, double] (0,-0.5*\R) arc (-90:-270:0.5*\R);
\end{tikzpicture} 
\end{aligned}
\ \ \ \ \ \ \ \ \ \ \ \ \ \ \ \ \ \ \ \ \ \ \ \ \ \ \ \ \ \ \ \ \ \ \ \ \ \ \ \ \ \ \ \ \longrightarrow\ \ \ \ \ \ \ \ \ Ob(\gG_{\E})
\end{eqnarray*}

\begin{eqnarray*}
\check{\pmb{s}}_{1}=(id,0,\eta): \ Mor(U_{\check{S}^{2}})=
\begin{aligned}
\begin{tikzpicture}
 \def\R{1} 
\shadedraw[shading=ball,ball color=red](0.4*\R,-0.3*\R) arc (-asin(0.6):-180+asin(0.6):0.5*\R) arc (180:360:0.4 and 0.06) -- cycle;
 \filldraw[gray, ultra nearly transparent] (0.4*\R,-0.3*\R) arc (-asin(0.6):180+asin(0.6):0.5*\R) arc (180:0:0.4 and 0.06)-- cycle;

\fill[gray, semitransparent] (0,-0.3*\R) ellipse (0.4 and 0.05); 
 \filldraw[fill=white, draw=black, fill opacity=0] (0.4*\R,-0.3*\R) circle (0.06);
 \filldraw[fill=white, draw=black, fill opacity=0] (0,-0.5*\R) circle (0.06);
 \draw[thick, double] (0.4*\R,-0.3*\R) arc (-asin(0.6):-90:0.5*\R);
\end{tikzpicture} 
\end{aligned}
\sqcup    
\begin{aligned}
\begin{tikzpicture}
\def\R{1} 
\shadedraw[shading=ball,ball color=red](0.4*\R,-0.3*\R) arc (-asin(0.6):-180+asin(0.6):0.5*\R) arc (180:360:0.4 and 0.06) -- cycle;
\filldraw[gray, ultra nearly transparent] (0.4*\R,-0.3*\R) arc (-asin(0.6):180+asin(0.6):0.5*\R) arc (180:0:0.4 and 0.06)-- cycle;

\fill[gray, semitransparent] (0,-0.3*\R) ellipse (0.4 and 0.05); 
\filldraw[fill=white, draw=black, fill opacity=0] (-0.4*\R,-0.3*\R) circle (0.06);
\filldraw[fill=white, draw=black, fill opacity=0] (0,-0.5*\R) circle (0.06);
\draw[thick, double] (0,-0.5*\R) arc (-90:-180+asin(0.6):0.5*\R);
\end{tikzpicture} 
\end{aligned}
\sqcup   
\begin{aligned}
\begin{tikzpicture}
\def\R{1} 
\shadedraw[shading=ball,ball color=red, white] (0,0) circle (0.5*\R);
\draw[thick, double] (0,0) circle (0.5*\R);
\end{tikzpicture} 
\end{aligned}
&& \sqcup\\
\begin{aligned}
\begin{tikzpicture}
 \def\R{1} 
\shadedraw[shading=ball,ball color=red](0.4*\R,-0.3*\R) arc (-asin(0.6):-180+asin(0.6):0.5*\R) arc (180:360:0.4 and 0.06) -- cycle;
 \filldraw[gray, ultra nearly transparent] (0.4*\R,-0.3*\R) arc (-asin(0.6):180+asin(0.6):0.5*\R) arc (180:0:0.4 and 0.06)-- cycle;

\fill[gray, semitransparent] (0,-0.3*\R) ellipse (0.4 and 0.05); 
 \filldraw[fill=white, draw=black, fill opacity=0] (0.4*\R,-0.3*\R) circle (0.06);
 \filldraw[fill=white, draw=black, fill opacity=0] (0,-0.5*\R) circle (0.06);
 \draw[thick, double] (0.4*\R,-0.3*\R) arc (-asin(0.6):-90:0.5*\R);
\end{tikzpicture} 
\end{aligned}
\sqcup
\begin{aligned}
\begin{tikzpicture}
\def\R{1} 
\shadedraw[shading=ball,ball color=red](0.4*\R,-0.3*\R) arc (-asin(0.6):-180+asin(0.6):0.5*\R) arc (180:360:0.4 and 0.06) -- cycle;
\filldraw[gray, ultra nearly transparent] (0.4*\R,-0.3*\R) arc (-asin(0.6):180+asin(0.6):0.5*\R) arc (180:0:0.4 and 0.06)-- cycle;

\fill[gray, semitransparent] (0,-0.3*\R) ellipse (0.4 and 0.05); 
\filldraw[fill=white, draw=black, fill opacity=0] (-0.4*\R,-0.3*\R) circle (0.06);
\filldraw[fill=white, draw=black, fill opacity=0] (0,-0.5*\R) circle (0.06);
\draw[thick, double] (0,-0.5*\R) arc (-90:-180+asin(0.6):0.5*\R);
\end{tikzpicture} 
\end{aligned}
&& \sqcup 
\begin{aligned}
\begin{tikzpicture}
\def\R{1} 
\shadedraw[shading=ball,ball color=red, white] (0,0) circle (0.5*\R);
\draw[thick, double] (0,0) circle (0.5*\R);
\end{tikzpicture} 
\end{aligned}
\ \ \ \ \ \ \ \ \ \ \ \ \ \ \ \ \ \ \ \longrightarrow\ \ \ \ \ Mor(\gG_{\E})\\
\sqcup
\begin{aligned}
\begin{tikzpicture}
 \def\R{1} 
\shadedraw[shading=ball,ball color=red](0.4*\R,-0.3*\R) arc (-asin(0.6):-180+asin(0.6):0.5*\R) arc (180:360:0.4 and 0.06) -- cycle;
 \filldraw[gray, ultra nearly transparent] (0.4*\R,-0.3*\R) arc (-asin(0.6):180+asin(0.6):0.5*\R) arc (180:0:0.4 and 0.06)-- cycle;

\fill[gray, semitransparent] (0,-0.3*\R) ellipse (0.4 and 0.05); 
\end{tikzpicture} 
\end{aligned}
&& \sqcup
\begin{aligned}
\begin{tikzpicture}
 \def\R{1} 
 \shadedraw[shading=ball,ball color=red, white] (0,0) circle (0.5*\R);
 \filldraw[fill=white, draw=black, fill opacity=0] (0,0.5*\R) circle (0.06);
 \filldraw[fill=white, draw=black, fill opacity=0] (0,-0.5*\R) circle (0.06);
\draw[thick, double] (0,0.5*\R) arc (90:-90:0.5*\R);
\end{tikzpicture} 
\end{aligned}
\sqcup   
\begin{aligned}
\begin{tikzpicture}
 \def\R{1} 
 \shadedraw[shading=ball,ball color=red, white] (0,0) circle (0.5*\R);
  \filldraw[fill=white, draw=black, fill opacity=0] (0,0.5*\R) circle (0.06);
 \filldraw[fill=white, draw=black, fill opacity=0] (0,-0.5*\R) circle (0.06);
 \draw[thick, double] (0,-0.5*\R) arc (-90:-270:0.5*\R);
\end{tikzpicture} 
\end{aligned}
\end{eqnarray*}
where 
\begin{eqnarray*}
\eta(re^{i\theta}\to re^{i\theta'})= && D(-v)\in G_{C} \ \ \ \ \ \  for \ \ re^{i\theta}\to re^{i\theta'}\in
\begin{aligned}
\begin{tikzpicture}
 \def\R{1} 
\shadedraw[shading=ball,ball color=red](0.5*\R,0) arc (0:180:0.5*\R) arc (180:360:0.5 and 0.1) -- cycle;
\fill[fill=gray,  semitransparent] (0,0) ellipse (0.5 and 0.1); 
 \fill[gray, ultra nearly transparent] (-0.5*\R,0) arc (-180:0:0.5*\R) arc (0:-180:0.5 and 0.1) -- cycle;
 \draw[gray, loosely dashed] (0.5*\R,0) arc (0:180:0.5 and 0.1);
 \filldraw[fill=white, draw=black, fill opacity=0] (0.5*\R,0) circle (0.06);
 \filldraw[fill=white, draw=black, fill opacity=0] (-0.5*\R,0) circle (0.06);
\draw[thick, double] (0.5*\R,0) arc (0:180:0.5*\R);
\end{tikzpicture} 
\end{aligned}
\ \ \ \ (\text{where} \ \theta'=\theta\pm2\pi)\\
&& Id\ \ \ \in G_{C} \ \ \ \ \ \ \ \ \  \text{everywhere else}.
\end{eqnarray*}
To extend  $\check{\pmb{s}}$ to $p$, we define $\pmb{s}$ on
$$
\inte_{m_{v}}\ltimes
\begin{aligned}
\begin{tikzpicture}
 \def\R{1} 
\shadedraw[shading=ball,ball color=red](0.4*\R,0.3*\R) arc (asin(0.6):180-asin(0.6):0.5*\R) arc (180:360:0.4 and 0.06) -- cycle;
 \fill[gray, ultra nearly transparent] (0.4*\R,0.3*\R) arc (asin(0.6):-180-asin(0.6):0.5*\R) arc (-180:0:0.4 and 0.06) -- cycle;
\fill[gray, semitransparent] (0,0.3*\R) ellipse (0.4 and 0.05); 
\end{tikzpicture}
\end{aligned}
 \twoarrows 
 G_{C}\ltimes
\begin{aligned}
\begin{tikzpicture}
 \def\R{1} 
\shadedraw[shading=ball,ball color=red](0.4*\R,0.3*\R) arc (asin(0.6):180-asin(0.6):0.5*\R) arc (180:360:0.4 and 0.06) -- cycle;
 \fill[gray, ultra nearly transparent] (0.4*\R,0.3*\R) arc (asin(0.6):-180-asin(0.6):0.5*\R) arc (-180:0:0.4 and 0.06) -- cycle;
\fill[gray, semitransparent] (0,0.3*\R) ellipse (0.4 and 0.05); 
\end{tikzpicture}
\end{aligned}
\times V_{C}
$$
by 
$$\pmb{s}_{0}(re^{i\theta})=(re^{i2\theta},0),$$
$$\pmb{s}_{1}(re^{i\theta} \xrightarrow{g} g \cdot re^{i\theta})=(re^{i2\theta},0)\xrightarrow{\mathfrak{i}(g)} (re^{i2\theta},0).$$
It is straightforward to verify the morphism above together with $\check{\pmb{s}}$ define a morphism from an orbisphere with $\inte_{m_{v}}$-orbipoint at $p$ to $\gG_{\E}$. 

This is the only representative orbifold morphism extends $\check{\pmb{s}}$. We refer the readers to \cite[Proposition 2.47]{TW} for a detailed reason in an explicit example.
\end{proof}

Now we describe the pullback bundle $\pmb{s}_{x}^{*} T\E_{v}$ which is the same as $\pmb{s}_{x}^{*} T\E_{v,C}$ since the image of $\pmb{s}_{x}$ is contained in $\E_{v,C}$. The pullback is obvious on $\check{S}^{2}$ since there is no orbipoint on the domain. The orbipoint north pole is covered by:
$$
\inte_{m_{v}}\ltimes \ \ \ \ \
\begin{aligned}
\begin{tikzpicture}
 \def\R{1} 
\shadedraw[shading=ball,ball color=red](0.4*\R,0.3*\R) arc (asin(0.6):180-asin(0.6):0.5*\R) arc (180:360:0.4 and 0.06) -- cycle;
 \fill[gray, ultra nearly transparent] (0.4*\R,0.3*\R) arc (asin(0.6):-180-asin(0.6):0.5*\R) arc (-180:0:0.4 and 0.06) -- cycle;
\fill[gray, semitransparent] (0,0.3*\R) ellipse (0.4 and 0.05); 
\end{tikzpicture}\end{aligned}
\times \com\times \com^{n} 
$$
with $g\cdot (z,\xi,\eta)=(g\cdot z,\xi,g\eta)$. 

Note that $g$ acts on the $\com$ component trivially since it acts on the horizontal (base) direction of $\E_{v,C}$ trivially. 

\begin{rmk}
We remark that the horizontal direction of the pullback bundle $\pmb{s}_{x}^{*}T\E_{v}$, is not the tangent bundle of $\com P^{1}(1,m_{v})$.
\end{rmk}

The following Lemma follows from a direct computation.
\begin{lem}\label{constsectwist}
The evaluation map of $\pmb{s}_{x}$ at the north pole (possibly an orbipoint) of $S^{2}$ lies in the twisted sector $\E_{v,(v^{-1})}$.
\end{lem}

Now we describe $(\pmb{s}_{x}^{*}T\E_{v})^{de}$, the desingularized bundle of $\pmb{s}_{x}^{*}T\E_{v}$. Let $\{i^{v}_{1},...,i^{v}_{n}\}\subset \{1,...,N\}$ be the subset such that  the cone spanned by $\{b_{i^{v}_{j}}|j=1,...,n\}$ over $\ration$ is the minimal cone $\sigma(v)$ containing $v$. Then $(\pmb{s}_{x}^{*}T\E_{v})^{de}$ is constructed by glueing two copies of the trivial vector bundles over disk $D\times (\com\times V_{c})\to D$ along the boundaries, where the glueing map $\rho:S^{1}\to Sp(\com\times V_{c},\omega_{0})$ is given by
$$
\rho(e^{i2\pi t})=diag(e^{i2\pi\cdot (-2t)},e^{i2\pi \lceil r^{v}_{1}\rceil t},...,e^{i2\pi \lceil r^{v}_{n} \rceil t}),
$$
where $\{r^{v}_{j}\}_{j=1}^{n}$ is defined by $v=\sum_{j=1}^{n} r^{v}_{j} b_{i^{v}_{j}}$, $\lceil r^{v}_{j}\rceil$ is the smallest integer no less than $r^{v}_{j}$.

\begin{lem}\label{desingsum}
\begin{enumerate}
\item\label{desingchern} The first Chern number of the desingularized bundle is
$$c_{1}((\pmb{s}_{x}^{*}T\E_{v})^{de} )=2-dim\sigma(v).$$
\item\label{summand} Each summand of $(\pmb{s}_{x}^{*}T\E_{v})^{de}$ has Chern number at least $-1$.
\end{enumerate}
\end{lem}
\begin{proof}
To show (\ref{desingchern}), note that the chern number can be computed from the Maslov index of the loop $\rho^{-1}$ in the group of symplectic matrices. The later is 
$$\mu(\rho^{-1})=2-\sum_{j=1}^{n} \lceil r^{v}_{j} \rceil=2-\sum_{ r^{v}_{j} \ne 0} 1=2-dim\sigma(v).$$

The desingularized bundle splits into line bundles whose first Chern numbers are 2 and $-\lceil r_{1}^{v} \rceil,..., -\lceil r_{n}^{v} \rceil$. Thus we have (\ref{summand}).
\end{proof}

The desingularized bundle of the vertical subbundle $\pmb{s}_{x}^{*}T^{vert}\E_{v}$ is the same as the vertical subbundle of the desingularized bundle $(\pmb{s}_{x}^{*}T\E_{v})^{de}$, which is constructed by glueing two copies of the trivial vector bundles over disk $D\times (\com\times V_{c})\to D$ along the boundaries, where the glueing map $\rho:S^{1}\to Sp(\com\times V_{c},\omega_{0})$ is given by
$$
\rho(e^{i2\pi t})=diag(e^{i2\pi \lceil r^{v}_{1}\rceil t},...,e^{i2\pi \lceil r^{v}_{n} \rceil t}).
$$

Now since the first Chern number of an orbibundle and the first Chern number of its desingularized bundle differ by the degree shifting number, we have:
\begin{lem}\label{chernclass}
$$c_1(\pmb{s}_{x}^{*}T\E_{v}) =2-dim\sigma(v)+\iota_{v^{-1}},$$
and
$$c_1(\pmb{s}_{x}^{*}T^{vert}\E_{v})=-dim\sigma(v)+\iota_{v^{-1}}.$$
Moreover when $m_{v}\neq 1$, $c_1(\pmb{s}_{x}^{*}T\E_{v})  =2-\iota_{v}$ and $c_1(\pmb{s}_{x}^{*}T^{vert}\E_{v})=-\iota_{v}$.
\end{lem}

In this paper, the Hamiltonian loops we consider is always Hamiltonian circle action in the sense of \cite{LM}. In this case, the orbifiber bundle has an alternative description: 

Let $S^{1}$ act on $S^{3}$ by $e^{i\theta}\cdot (z_{1},z_{2}):=(e^{i\theta}z_{1},e^{i\theta}z_{2})$, where $|z_{1}|^{2}+|z_{2}|^{2}=1$. Then $\E_{v}=S^{3}\times \X/S^{1}$, where the $S^{1}$ action on $\X$ corresponds to the lattice point $v$, i.e. generated by $H_{v}$.

To see that the quotient construction is equivalent to the glueing construction, we need the following lemma:
\begin{lem}\label{S3atlas}
For any atlas $\{U_{\alpha}\}$ of $S^{2}$, there exists an $S^{1}$-invariant atlas $\{\tilde{U}_{\alpha}\}$ of $S^{3}$ such that $U_{\alpha}=\tilde{U}_{\alpha}/S^{1}$.
\end{lem}

\begin{proof}
Let $\chi:S^3\to S^2$ be the Hopf map. An atlas of $S^{2}$ determines an atlas of $S^{3}$ by pullback of $\chi$. This is the atlas with the required property, since $S^{1}$ acts on $S^{3}$ preserving the fibers.
\end{proof}

Now the abstract construction of $\X_{v}$ by glueing stacks can be carried out by glueing Lie groupoids\footnote{See \cite[Appendix]{TW} for definitions and details} as the following. Let $\gG_{\X}$ be the translation groupoid $G\ltimes Z$. The map $\tilde{\pmb{\gamma}}_{v}$ associated to the Hamiltonian loop is represented by $\gamma:U_{S^{1}}\times \gG_{\X}\to \gG_{\X}$ where $U_{S^{1}}$ is the groupoid determined by an atlas of $S^{1}$. Take an atlas of $S^{2}$ which gives arise the atlas of $S^{1}$ when cut along the equator, denote the corresponding groupoid chart of $S^{2}$ to be $U_{S^{2}}$ and $U_{S^{2}}^{+}$, $U_{S^{2}}^{-}$ the groupoid charts of the two half disks. Then $\E_{v}$ is represented by the groupoid glued from $U_{S^{2}}^{+}\times \gG_{\X}$ and $U_{S^{2}}^{-}\times \gG_{\X}$ using $\gamma$. A concrete example of this construction can be found in \cite{TW}.

Note that Lemma \ref{S3atlas} defines an atlas of $S^{3}$ for the atlas of $S^{2}$ used above. Denote $U_{S^{3}}$ as the groupoid chart of $S^{3}$ determined by this atlas. It is straightforward to construct an Lie groupoid isomorphism $$U_{S^3}\times \gG_{\X}/S^{1}\to U_{S^{2}}^{+}\times \X\sqcup U_{S^{2}}^{-}\times \X/\sim_{\gamma}.$$ 
Passing to stacks we have an diffeomorphism from $S^{3}\times \X/S^{1}$ to $\E_{v}$. Consequently there is an obvious orbifold morphism from $S^{3}\times \X$ to $\E_{v}$ determined by the quotient. We denote the morphism as $pr:S^{3}\times \X\to \E_{v}$. 

The coupling form $\textbf{u}_{v}$  has the following description:
Let $\alpha \in \Omega^1(S^3)$ be the usual contact form on the unit sphere, normalized so that  $d\alpha= \chi^*(\tau)$ where $\tau$ is the standard area form on $S^2$ with total area $1$. Then,
\begin{equation}\label{coupling}
    \textbf{u}_{v} = pr_*(\omega - dH_{v}\alpha).
\end{equation}
From the definition of the coupling class $\mathbf{u}_{v}$, it is easy to check the following:

\begin{lem}\label{couplingconst}
$$\mathbf{u}_{v}([s_x])=-H_{v}(x).$$
\end{lem}

\begin{lem}\label{couplingclass}
$\mathbf{u}_{v}(\sigma+\iota_*B)=\mathbf{u}_{v}(\sigma)+ \omega(B)$  where $\sigma\in H^{sec}_{2}(|\E_{v}|,\inte),\ B\in H_2(|\mathcal{X}|,\inte)$, $\iota:\mathcal{X}\to\E_{v}$ is the inclusion of a fiber at the north pole and $\iota_{*}:H_2(|\mathcal{X}|,\inte)\to H_2(|\mathcal{E}_{v}|,\inte)$ is the induced pushforward map.
\end{lem}

\section{Seidel Elements}

\subsection{Review of Seidel Representation}\label{sec:seidel}

First we recall the notion of Novikov ring.

\begin{defn}\label{dfn:Novikov}
Define a ring $\Lambda^{univ}$ as $$\Lambda^{univ}=\left\{\sum_{k\in\mathbb{R}}r_k T^k|r_k\in\mathbb{Q}, \#\{k<c|r_k\neq 0\}<\infty \,\,\forall c\in\mathbb{R}\right\}$$ and equip it with a grading given by $deg(T)=0$.

Let $\mathcal{C}$ be the Mori cone of $\X$ which is a finitely generated monoid. Then there is a maximal fraction $1/a$ with $a\in \inte_{+}$ such that $c_{1}(T\X)(\mathcal{C})\subset \ration$ is contained in the monoid generated by $1/a$.
Define $\Lambda:=\Lambda^{univ}[q^{\frac{1}{a}},q^{-\frac{1}{a}}]$  with the grading given  by $deg(q)=2/a$. 
\end{defn}

The Gromov-Witten theory of symplectic orbifolds is constructed by Chen-Ruan in \cite{CR2}, to which we refer the readers for more details. Let $$\<\alpha_{1},\alpha_{2},\alpha_{3} \>_{0,3,A}$$ be the $3$-point genus $0$ degree $A$ Gromov-Witten invariants of $\X$ with insertions $\alpha_1, \alpha_2, \alpha_3\in H^*(I\X, \mathbb{Q})$. We may assemble these genus zero orbifold Gromov-Witten invariants with 3 marked points using the Novikov ring: $$\<\alpha_{1},\alpha_{2},\alpha_{3} \>:=\sum_{A\in H_{2}(|\X|,\inte)}\<\alpha_{1},\alpha_{2},\alpha_{3} \>_{0,3,A}q^{c_{1}[A]} t^{\omega[A]}.$$ This is used to define the {\em quantum product},  an associative multiplication $*$ on 
$H^{*}(I\X,\mathbb{Q})\otimes \Lambda$, as follows: $$\<\alpha_{1}*\alpha_{2},\alpha_{3}\>_{orb}:=\<\alpha_{1},\alpha_{2},\alpha_{3} \>,\ \ \text{for}\ \alpha_{i}\in H^{*}(I\X,\mathbb{Q}).$$ The resulting ring, denoted by $QH_{orb}^{*}(\X,\Lambda)$, is called the orbifold quantum cohomology ring of $(\X,\Omega)$. 

Let $QH^*_{orb}(\X,\Lambda)^\times$ be the group of invertible elements (with respect to the quantum product "$*$") in $QH_{orb}^*(\X,\Lambda)$. In \cite{TW}, the authors construct a group homomorphism: 
$$\mathcal{S}: \pi_1(Ham(\X,\omega))\to QH^*_{orb}(\X, \Lambda)^\times.$$
Generalizing the manifold case, this is called the Seidel representation for symplectic orbifold $(\X,\omega)$. We briefly explain its construction. Represent a homotopy class $a\in \pi_{1}(Ham(\X,\omega))$ by a Hamiltonian loop $\pmb{\gamma}$, then we can construct Hamiltonian orbifiber bundle $\E_{\gamma}$ as in Section \ref{orbifiberSec}. Let $\{f_i\}$ be an additive basis of $H^*(I\X)$, $\{f^i\}$ another additive basis of $H^*(I\X)$ dual to $\{f_j\}$ with respect to the orbifold Poincar\'e pairing. Denote $c_1(T^{vert}\E)$ by $c_1^{vert}$. Let $\iota$ be an inclusion of a fiber over a point in $S^2$ (we choose the north pole throughout this paper). There is a Gysin map induced by this inclusion: $\iota_{*}: H^{*}(I\X,\mathbb{Q})\to H^{*+2}(I\E_{\gamma},\mathbb{Q})$. One can think of this map as a union of maps  from $H^{*}(\X_{(g)},\mathbb{Q})$ to $H^{*+2}(\E_{\gamma,(g)},\mathbb{Q})$, which makes sense because there is not orbifoldness along the horizontal direction.

\begin{defn}
Seidel representation for a symplectic orbifold $(\X,\omega)$ is defined as:
\begin{equation}\mathcal{S}(a):=\sum_{\sigma\in H_2^{sec}(|\E_{\gamma}|,\inte)} \left(\sum_{i} \< \iota_*f_i\>^{\E_{\gamma}}_{0,1,\sigma}f^i\right) \otimes q^{c_1^{vert}(\sigma)}t^{\textbf{u}_{\gamma}(\sigma)}.
\end{equation}
\end{defn}

The definition of $\mathcal{S}$ does not depend on the choice of Hamiltonian loop $\gamma$ representing the homotopy class $a$, thus is a well-defined map. Moreover it is a group homomorphism:

\begin{thm}[\cite{TW}, Theorem 1.2]
The map $\mathcal{S}$ has the following properties.
\begin{enumerate}
\item \bf{Triviality}:
\begin{equation}\label{triviality}
S(e)=1;
\end{equation}
\item \bf{Composition}:
\begin{equation}\label{composition}
\mathcal{S}(a\cdot b)=\mathcal{S}(a)*\mathcal{S}(b).
\end{equation}
\end{enumerate}
\end{thm}

\subsection{Seidel Element for Toric Orbifolds}\label{seidelelement}

In this section we consider a symplectic toric orbifold $\X$ associated to a labelled moment polytope:
$$
\Delta= \bigcap_{i=1}^N\{\alpha\in\textbf{t}^*| \<\alpha,b_i\>\le \lambda_i\},
$$
with $m_{i}$ the labeling number on the $i$-th facet, $y_{i}$ the primitive outward normal vector and $b_{i}=m_{i}y_{i}$.

There is a naturally defined complex structure $\check{J}$ induced from the complex structure on $\com^{N}$. For any orbifiber bundle $\E\to S^2$ considered in Section \ref{orbifiberSec}, this complex structure $\check{J}$ determines a complex structure $J$ on the total orbifold $\E$, such that the projection $\pi:\E\to S^{2}$ is $j$-$J$ holomorphic, where $j$ is the complex structure on $S^{2}$ when identified with $\com P^{1}$. 

Recall from Section \ref{toricQH} that the Chen-Ruan cohomology of a toric orbifold can be expressed as a quotient of a polynomial ring $\ration[X_1,X_2,...,X_M] $.
Define the following valuation 
\begin{equation}\label{dfn:valuation}
\begin{split}
&\mathfrak{v}_{T}:\ration[X_1,X_2,...,X_M] \otimes\Lambda \to \ration;\\ 
&\mathfrak{v}_{T}(\sum_{d,k}a_{d,k}\otimes q^d T^k)=min\{k|\exists d: d_{d.k}\neq 0\}.
\end{split}
\end{equation}
This induces a valuation on the quantum cohomology $QH_{orb}^{*}(\X,\omega)$ which we still denote as $\mathfrak{v}_{T}$ when there is no ambiguity.

In this section we compute the Seidel elements for Hamiltonian loops determined by lattice points $y_{k}\in\text{Gen}(\Sigma)$, $k=1,...,M$. Let $H_k: \X\to \mathbb{R}$ be the Hamiltonian associated to $y_k$. Let $\E^{k}$ be the Hamiltonian bundle associated to $y_{k}$. Let $a_{k}\in \pi_{1}(Ham(\X,\omega))$ be the homotopy class of the Hamiltonian loop generated by $y_{k}$. Denote by $\sigma(y_{k})$ the minimal cone in $\Sigma$ containing $y_{k}$, and $y_{k}=\sum_{b_{i}\in \sigma(y_{k})} r_{ki} b_{i}$.
Let $h.o.t.(T^{r})$ be terms of order $>r$ with respect to the valuation $\mathfrak{v}_{T}$ on $QH_{orb}^{*}(\X,\omega)$. Then
\begin{thm}\label{seidelele}
$$\mathcal{S}_{k}:=\mathcal{S}(\alpha_{k})=X_{k}\otimes q^{-\sum_{b_{i}\in \sigma(y_{k})} r_{ki}} T^{-\sum_{b_{i}\in \sigma(y_{k})} r_{ki} \lambda_{i}}+h.o.t.(T^{-\sum_{b_{i}\in \sigma(y_{k})} r_{ki} \lambda_{i}}).$$
\end{thm}
\begin{proof}
All we need is to compute $\< \iota_*f_i\>^{\E_{\gamma}}_{0,1,\sigma}$. Denote $\textbf{x}_{k}=(\E^{k}_{(y_{k}^{-1})})$. Let $\mathcal{F}_{max}$ be the (non-effective) suborbifold of $\mathcal{X}$ on which $ H_{k}$ is maximized. This suborbifold is fixed by the Hamiltonian loop. Every $|x|\in |\mathcal{F}_{max}|$ defines a section class $\sigma_{x}\in H_{2}(\E^{k},\inte)$ of the topological bundle $|\E^{k}|\to S^{2}$. It is easy to check that $\sigma_x$ does not depend on the choice of $|x|\in |\mathcal{F}_{max}|$, so we can denote the homology class as $\sigma_{max}$. Note that $\mathcal{F}_{max}$ swipes out a suborbifold $S\mathcal{F}_{max}\subset \E_{k}$ which is also fibered over $S^{2}$. It also determines a submanifold $S\mathcal{F}_{max,(g)}$ in each stratum $\E^{k}_{(g)}$ of $I\E^{k}$, which is a topological bundle over $S^{2}$ with fiber $|\mathcal{F}_{max,(g)}|$.

The computation of $\< \iota_*f_i\>^{\E_{\gamma}}_{0,1,\sigma}$ is divided into the following steps:

\begin{enumerate}
\item[Step 1:] Every element in $\overline{\mathcal{M}}_{0,1}(\E^{k},\sigma_{max},J,\textbf{x}_{k})$ is represented by a constant sectional morphism. This follows from the same computation as \cite[Lemma 3.1]{MT}.
\item[Step 2:] The domain of the constant sectional morphism is $\com P(1, m_{k})$, with $m_{k}=ord(D(-y_{k}))$. This is a direct conclusion from Proposition \ref{constantsection}.
\item[Step 3:] The constant sectional morphism is Fredholm regular.

Before prove Step 3, we run a quick check by dimension formula:
\begin{eqnarray*}
vdim \overline{\mathcal{M}}_{0,1}(\E^{k},\sigma_{max},J,\textbf{x}_{k}) & = & dim_\mathbb{R}\E^{k}+2c_{1}(T\E^{k})(\sigma_{max})+2-2\iota_{y_{k}^{-1}}-6\\
& = & 2(n+1)+2(2-dim\sigma(y_{k})+\iota_{y_{k}^{-1}})+2-2\iota_{y_{k}^{-1}}-6\\
& = & 2(n-dim\sigma(y_{k}))+2.
\end{eqnarray*}    

On the other hand, by the above discussion $\overline{\mathcal{M}}_{0,1}(\E^{k},\sigma_x,J,\textbf{x}_{k})$ can be identified with $S\mathcal{F}_{max,(y_{k}^{-1})}$. So its dimension is $2n-codim \Phi(\mathcal{F}_{max,(y_{k}^{-1})})+2=2n-2dim\sigma(y_{k})+2$. Therefore the virtual dimension and the actual dimension match. This gives an evidence for Fredholm regularity of constant sectional morphisms. 

Let $p: \com P^{1}(1,m_{k}))\to \com P^{1}$ be the coarse moduli space map. Then $p_*\pmb{s}^*T\E^{k}$ is the desingularized bundle $(\pmb{s}^*T\E^{k})^{de}$ of $\pmb{s}^*T\E^{k}$. This vector bundle over $\com P^{1}$ splits into a direct  sum of line bundles, as discussed in Lemma \ref{desingsum}.

\begin{lem}\label{reglem}
If each summand of $p_*\pmb{s}^*T\E^{k}$ has Chern number at least $-1$, then the linearized $\bar{\partial}_{J}$ operator $D_{\pmb{s}}\bar{\partial}_{J}$ is onto.
\end{lem}
\begin{proof}
To show that $D_{\pmb{s}}$ is onto we need the cohomology group $H^1(\pmb{s}^*T\E^{k})$ to vanish. By a general property of $p$, we have $H^1(\pmb{s}^*T\E^{k})=H^1(p_*\pmb{s}^*T\E^{k})$. Since $p_*\pmb{s}^*T\E^{k}$ splits, $H^1(p_*\pmb{s}^*T\E^{k})$ splits into a direct sum of $H^1$ of ths summands. Let L be any summand of $p_*\pmb{s}^*T\E^{k}$. We need to show that the cohomology group $H^1(L)$ vanishes. By Serre duality this group is isomorphic to $H^0(L^*\otimes K)^*$ where $K$ is the canonical line bundle of $\com P^1$. We need $c_1(L^*\otimes K)<0$ in order for this group to vanish. Now $$c_1(L^*\otimes K)=-c_1(L)+c_1(K)=-c_1(L)-1-1.$$ So we need $-c_1(L)-1-1<0$, namely $c_1(L)>-2$, i.e. $c_1(L)\geq -1$.
\end{proof}
The above lemma together with Lemma \ref{desingsum} complete the proof of Step 3, i.e., the moduli space $\Mbar_{0,1}(\E^{k},\sigma_x,J,\textbf{x}_{k})$ is regular.

\item[Step 4:] The image of $\overline{\mathcal{M}}_{0,1}(\E^{k},\sigma_{max},J,\textbf{x}_{k})$ under the evaluation lies in the twisted sector $\E^{k}_{(y_{k}^{-1})}$. This follows from Lemma \ref{constsectwist}.
\item[Step 5:] Show that $\sum_i\<\iota_*f_i\>^{\E^{k}}_{0,1,\sigma_{max}}f^i=X_{i}=\lambda^{y_{i}}\in H^{*}_{CR}(\X,\ration)$.

From the previous steps, we have $ev_{*}[\Mbar_{0,1}(\E^{k},\sigma_x,J,\textbf{x}_{k})]=[S\mathcal{F}_{max,(y_{k}^{-1})}]$ as cycles in $\mathcal{I}\E^{k}$. Then 
\begin{eqnarray*}
\sum_i \<\iota_*f_i\>^{\E^{k}}_{0,1,\sigma_{max}}f^i & = & \sum_{i}(\int_{ev_{*}[\overline{\mathcal{M}}_{0,1}(\E^{k},\sigma_{max},J,\textbf{x}_{k})]} \iota_*f_i) \ f^i\\
& = & \sum_{i:f_i\in H^*(\mathcal{X}_{(y_{k}^{-1})},\ration)}(\int_{ev_{*}[\overline{\mathcal{M}}_{0,1}(\E^{k},\sigma_{max},J,\textbf{x}_{k})]} \iota_*f_i) \ f^i\\
& = & \sum_{i:f_i\in H^*(\mathcal{X}_{(y_{k}^{-1})},\ration)}(\int_{[\mathcal{F}_{max,(y_{k}^{-1})}]} f_i)\ f^i \\
& = & X_{i}
\end{eqnarray*}
\item[Step 6:] Any $J$-holomorphic non-constant sectional morphism $\pmb{s}$ or a constant sectional morphism constructed from a point not in $\mathcal{F}_{max}$ satisfies 
$$\mathbf{u}_{y_{k}}([\pmb{s}])>\mathbf{u}_{y_{k}}(\sigma_{max})=-max\ H_{k}=-\sum_{b_{i}\in \sigma(y_{k})} r_{ki} \lambda_{i}.$$

If $\pmb{s}$ is a constant sectional morphism determined by a point $x$ with $H(x)>max\ H$, by Lemma \ref{couplingconst}, $\mathbf{u}_{y_{k}}([\pmb{s}])=-H_{k}(x)>-max\ H_{k}=\mathbf{u}_{y_{k}}(\sigma_{max})$.

If $\pmb{s}$ is a non-constant sectional morphism, compute $\mathbf{u}_{y_{k}}([\pmb{s}])$ by integrating the pullback of $\mathbf{u}_{y_{k}}$. Recall from (\ref{coupling}) that the coupling form is given by  $\textbf{u}_{y_{k}}=\omega-d H_{k}\alpha$. We choose the complex structure $\check{J}$ on $\X$ which is induced from the complex structure on $\com^{N}$. Consequently $\check{J}$ is invariant under the Hamiltonian circle action generated by $y_{k}$. Let $\hat{j}$ be the standard complex structure on $S^{3}$ induced from $\com^{2}$. Every non-zero tangent vector $\xi\in T_{[z,x]} S^{3}\times \X/S^{1}$ can be uniquely represented by a vector $\eta+v\in T_{(z,x)} S^{3}\times \X$ with $\eta\in ker\alpha$ and $v\in T_{x}\X$. Then 
$$\textbf{u}_{y_{k}}(\xi,J\xi)= \omega(v,\check{J}v)-dH_{k}\alpha(\eta+v,\hat{j}\eta+\check{J}v)=\omega(v,\check{J}v)-H_{k}d\alpha(\eta,\hat{j}\eta)>-max\ H_{k}.$$
Integrating over $S^{2}$ we get $\mathbf{u}_{y_{k}}([\pmb{s}])>\mathbf{u}_{y_{k}}(\sigma_{max})$.
\end{enumerate}

Thus we complete the proof of Theorem \ref{seidelele}.
\end{proof}

\begin{defn}
For a toric orbifold $\X$, the {\em reduced Seidel elements} are defined as:
$$\tilde{\mathcal{S}}_{k}: =\mathcal{S}_{k}\otimes q^{\sum_{b_{i}\in \sigma(y_{k})} r_{ki}} T^{\sum_{b_{i}\in \sigma(y_{k})} r_{ki} \lambda_{i}}$$
\end{defn}

\begin{cor}\label{highorder}
$\tilde{\mathcal{S}}_{k}=X_{k}+h.o.t.(T^{0}).$
\end{cor}

\section{Quantum Cohomology of Toric Orbifolds}
In this section we give an explicit description of the quantum cohomology ring of toric orbifolds using the Seidel elements computed in Section \ref{seidelelement}.

\subsection{Main Result}\label{sec:results}
Recall that for $k=1,...,N$, $y_{k}$ is the primitive vector of the $k$-th ray, and for $k=N+1,...,M$, $y_{k}\in \text{Gen}(\Sigma)$. Let $\sigma(y_{k})$ be the minimal cone in $\Sigma$ containing $y_{k}$, we have $y_{k}=\sum_{b_{i}\in \sigma(y_{k})} r_{ki} b_{i}$, where $r_{ki}$ are positive rational numbers. Note that when $k=1,...,N$, $y_{k}= \frac{1}{m_{k}}b_{k}$.

For a generalized primitive collection $I\subset \{1,...,M\}$, let $\sigma(I)$ be the minimal cone in $\Sigma$ containing $\sum_{k\in I} y_{k}$, then there exists $\{c_{j}\}\subset \inte_{+}$ such that
 $$\sum_{k\in I} y_{k} =\sum_{y_{j}\in \sigma(I)} c_{j} y_{j}$$
 since the fan is complete.
 
 Write both sides in terms of $b_{i}$'s:
 $$\sum_{k\in I} \sum_{b_{i}\in \sigma(y_{k})} r_{ki} b_{i}=\sum_{y_{j}\in \sigma(I)} c_{j} \sum_{b_{i}\in \sigma(y_{j})} r_{ji} b_{i}=\sum_{y_{j}\in \sigma(I)} \sum_{b_{i}\in \sigma(y_{j})} c_{j} r_{ji} b_{i}.$$

Let 
\begin{eqnarray}\label{COmega}
C_{I}& & =\sum_{k\in I} \sum_{b_{i}\in \sigma(y_{k})} r_{ki} -\sum_{y_{j}\in \sigma(I)} \sum_{b_{i}\in \sigma(y_{j})} c_{j} r_{ji},\\
\Omega_{I} & & =\sum_{k\in I} \sum_{b_{i}\in \sigma(y_{k})} r_{ki} \lambda_{i} -\sum_{y_{j}\in \sigma(I)} \sum_{b_{i}\in \sigma(y_{j})}c_{j}  r_{ji}\lambda_{i}.
\end{eqnarray}

\begin{lem}\label{effective} For $\Omega_{I}$ defined as above, we have $\Omega_{I}>0$.
\end{lem}
\begin{proof}
Define a piecewise linear function $\phi_{\omega}:\textbf{t}\to \real$ as:
 $$\phi_{\omega}(u):=\sum_{b_{i}\in\sigma} -\lambda_{i} \langle b_{i}^{\vee},u\rangle,\ \ \ if\ u\in\sigma,$$
 for $\sigma$ any full dimensional cone of the fan $\Sigma$ and $\{b_{i}^{\vee}\}$ the dual basis of $\{b_{i}|b_{i}\in\sigma\}$. 
 
Then 
\begin{equation}\label{convexfunc}
\Omega_{I}=-\sum_{k\in I} \phi_{\omega}(y_{k}) +\phi_{\omega}(\sum_{y_{j}\in \sigma(I)} c_{j} y_{j}).
\end{equation}

On the other hand, $\phi_{\omega}$ corresponds to the symplectic form $\omega$ of the toric orbifold $\X$ under the isomorphism between $H^{2}(\X,\real)\cong \real[b_{1}^{\vee},...,b_{N}^{\vee}]/\<\sum_{i=1}^{n}\theta(b_{i})b_{i}^{\vee}, \theta\in \mathbf{M}\>$. Since $\omega$ lies in the K\"{a}hler cone of $\X$, $\phi_{\omega}$ is a strictly convex function (in the sense that if $v_{1},v_{2}$ are not contained in the same cone, then $\phi_{\omega}(v_{1}+v_{2})>\phi_{\omega}(v_{1})+\phi_{\omega}(v_{2})$). Together with (\ref{convexfunc}), we have $\Omega_{I}>0$.
\end{proof}

\begin{defn}\label{dfn:QSR}
The {\em quantum Stanley-Reisner relation} associated with the generalized primitive collection $I$ is defined to be
\begin{equation}
\prod_{k\in I} Z_{k}-q^{C_{I}}T^{\Omega_{I}} \prod_{y_{j}\in \sigma(I)} Z_{j}^{c_{j}} =0.
\end{equation}
Let $SR_\omega$ be the ideal generated by expressions as the left-hand side of the above equation. $SR_\omega$ is called the {\em quantum Stanley-Reisner ideal}. 
\end{defn} 

 From the composition property of Seidel representation, we have 
 \begin{thm}\label{qSR_seidel}
 The reduced Seidel elements $\tilde{\mathcal{S}}_{k}$, $k=1,...,M$, satisfy the quantum Stanley-Reisner relations and the cone relations.
 \end{thm}
 
Let 
\begin{eqnarray*}
\mathfrak{P}_{\xi}^{(0)} \negthickspace\negthickspace\negthickspace\negthickspace\negthickspace & & :=\sum_{i=1}^{N}\langle e_{\xi},b_{i} \rangle X_{i}^{m_{i}}\ ,\ \ \ \ \  \xi=1,...,n,\ \text{where}\ \{e_{\xi}\}_{\xi=1}^{n}\ \text{is a basis of}\ \mathbf{M}.
\end{eqnarray*}

\begin{thm}\label{thm:QH}
For $\xi=1,...,n$, there exist $\mathfrak{P}_{\xi}\in \Lambda[X_{1},...,X_{M}]$ such that
\begin{enumerate}
\item $\mathfrak{P}_{\xi}(\tilde{\mathcal{S}}_{1},...,\tilde{\mathcal{S}}_{M})=0$ in $H^*(I\X, \Lambda)$;
\item $\mathfrak{P}_{\xi}= \mathfrak{P}^{(0)}_{\xi} + h.o.t.(T^{0})$;
\item The ring homomorphism 
$$\Psi: \frac{\Lambda[X_1,...,X_M]}{Clos_{\mathfrak{v}_{T}}(\<\mathfrak{P}_{\xi} |\xi=1,...,n\> + SR_{\omega} + \mathcal{J}(\Sigma))} \to QH_{orb}^*(\X,\Lambda), \quad X_i\mapsto \tilde{\mathcal{S}}_k$$  is an isomorphism.
\end{enumerate}

\end{thm}
\begin{proof}
Let $\hat{\Psi}: \Lambda[X_1,...,X_M]\otimes \Lambda\to QH_{orb}^*(\X,\Lambda)$ be the map sending $X_i$ to $\tilde{\mathcal{S}}_{i}$. Let $\delta_{\X}$ be a positive number such that the symplectic area of any non-constant $J$-holomorphic curve is bounded below by $\delta_{\X}$. 

We first show that $\hat{\Psi}$ is surjective. For any $\alpha\in QH_{orb}^*(\X,\Lambda)$, Let $L(\alpha) $ be the leading term of $\alpha$, namely $\mathfrak{v}_{T}(L(\alpha))=\mathfrak{v}_{T}(\alpha)$ and $\mathfrak{v}_{T}(L(\alpha)-\alpha)>\mathfrak{v}_{T}(\alpha)$. Then $L(\alpha)=\sum_{s}L^{\alpha}_{s}(X_{1},...,X_{M})q^{s}$, and $\hat{\Psi}(L(\alpha))=\sum_{s}L^{\alpha}_{s}(\tilde{S}_{1},...,\tilde{S}_{M})q^{s}$. Define $\alpha_{1}=\alpha-\hat{\Psi}(L(\alpha))$. The leading terms in $\alpha$ and $\hat{\Psi}(L(\alpha))$  cancell, so $\mathfrak{v}_{T}(\alpha_{1})\ge \mathfrak{v}_{T}(\alpha)+\delta_{\X}$. We repeat the above procedure using $\alpha_{1}$, and continue the argument inductively, then $\mathfrak{v}_{T}(\alpha_{k})\to \infty$, and $$\hat{\Psi}\left(\lim_{k\to \infty} \left(L(\alpha)+L(\alpha_{1})+L(\alpha_{2})+...+L(\alpha_{k})\right)\right)=\alpha.$$ This proves surjectivity.

Next we show the kernel of $\hat{\Psi}$ is 
$$Clos_{\mathfrak{v}_{T}}(\<\mathfrak{P}_{\xi} |\xi=1,...,n\> + SR_{\omega} + \<\mathfrak{Q}_{\eta} |\eta=1,...,M-N\>).$$

 We need the following lemma which is analogous to \cite[Lemma 5.1]{MT}:
\begin{lem}\label{alglemma}
Let
\begin{eqnarray*}
\hat{\phi} & : & \ration[X_1,X_2,...,X_M] \to H_{CR}^*(\mathcal{X},\ration)\\
\hat{\Phi} & : & \ration[X_1,X_2,...,X_M] \otimes\Lambda \to QH_{orb}^*(\X,\Lambda)
\end{eqnarray*}
be ring homomorphisms such that $\hat{\phi}(X_i)=\hat{\Phi}(X_i)=X_i$ for $i=1,...,M$.
 Let $w_1,...w_m\in \ration[X_1,X_2,...,X_M]$ generate the kernel of $\hat{\phi}$, and suppose $v_1,...v_m\in Ker \hat{\Phi}$ and 
 \begin{equation}\label{highordcond}
 \mathfrak{v}_{T}(w_i-v_i)>0\ \text{for all}\ i.
\end{equation}

 Then the kernel of $ \hat{\Phi}$ is $$Ker\hat{\Phi}=Clos_{\mathfrak{v}_{T}}(\<v_1,...,v_m\>).$$
\end{lem}
\begin{proof}
Since in the orbifold case there is still a universal lower bound for symplectic area of non-constant $J$-holomorphic curves, thus we have \cite[Lemma 5.1]{MT} for orbifold quantum cohomology without any modification of the proof.
\end{proof}

From Lemma \ref{CRring}, we know that $Ker \hat{\phi}$ is generated by the Stanley-Reisner relations, cone relations, and $\mathfrak{P}_{\xi}^{(0)}$'s. We have shown in Theorem \ref{qSR_seidel} that the quantum Stanley-Reisner relations and cone relations lie in the kernel of $\hat{\Phi}$. Moreover they satisfy (\ref{highordcond}).

Now we construct $\mathfrak{P}_{\xi}$'s which lie in the Kernel of $\hat{\Phi}$ and satisfy the high order condition (\ref{highordcond}).

Let $\lambda_{0}=0$.
By Corollary \ref{highorder}, there is a $\lambda_{1}\ge\lambda_{0}+\delta_{\X}$, such that
\begin{equation}\label{Pinduction}
\mathfrak{P}_{\xi}^{(0)}(\tilde{\mathcal{S}}_{1},...,\tilde{\mathcal{S}}_{M})=\sum_{i=1}^{n}\langle e_{\xi},b_{i} \rangle \tilde{\mathcal{S}}_{i}=\sum_{i=1}^{n}\langle e_{\xi},b_{i} \rangle X_{i}^{m_{i}} +R^{1}\ T^{\lambda_{1}}+h.o.t.(T^{\lambda_{1}}),
\end{equation}
with $\mathfrak{v}_{T}(R^{1})=0$.
Then define $\mathfrak{P}_{\xi}^{(1)}=\mathfrak{P}_{\xi}^{(0)}-R^{1}\ T^{\lambda_{1}}$. Since $R^{1}=\sum_{s}R^{1}_{s}(X_{1},...,X_{M})q^{s}$, $\mathfrak{P}_{\xi}^{1}$ is again a polynomial in $X_{1},...,X_{M}$ with coefficients in $\Lambda$. Replace $\mathfrak{P}_{\xi}^{(0)}$ with $\mathfrak{P}_{\xi}^{(1)}$ in (\ref{Pinduction}), we get 
\begin{eqnarray*}\label{Pinduction2}
\mathfrak{P}_{\xi}^{(1)}(\tilde{\mathcal{S}}_{1},...,\tilde{\mathcal{S}}_{M}) & & =\sum_{i=1}^{n}\langle e_{\xi},b_{i} \rangle X_{i}^{m_{i}} +R^{1}(X_{1},...,X_{M})\ T^{\lambda_{1}}+h.o.t-(R^{1}(\tilde{S}_{1},...,\tilde{S}_{M})+h.o.t.(T^{\lambda_{2}}))\\
& & =\sum_{i=1}^{n}\langle e_{\xi},b_{i} \rangle X_{i}^{m_{i}} +R^{2}\ T^{\lambda_{2}}+h.o.t.(T^{\lambda_{2}}),
\end{eqnarray*}
for some $\lambda_{2}>\lambda_{1}+\delta_{\X}$.

Then we can construct $\mathfrak{P}_{\xi}^{(2)}$ as before. Continue the procedure inductively, we construct $\mathfrak{P}_{\xi}^{(k)}$, $k=0,1,2,....$. Then define $\mathfrak{P}_{\xi}=\lim_{k\to\infty} \mathfrak{P}_{\xi}^{(k)}$.

By Lemma \ref{alglemma}, we conclude that $Ker \hat{\Phi}$ is generated by quantum Stanley-Reisner relations, cone relations, and $\<\mathfrak{P}_{\xi} |\xi=1,...,n\>$. This completes the proof.
\end{proof}
\begin{rmk}
We remark that the quantum cohomology ring can also be described as some sort of quantization of the group ring $\ration[\mathbf{N}]^{\Sigfan}$ by  reversing the procedure in Section \ref{toricQH}. Let $\ration[\mathbf{N}]^{\Sigfan}_{\omega}$ be the ring with the same elements as $\ration[\mathbf{N}]^{\Sigfan}\otimes \Lambda$ and the product $*$ defined as the following: 
 $$\lambda^{e_1}*\lambda^{e_2}:=q^{C(e_1,e_2)}T^{\Omega(e_1,e_2)}\lambda^e,$$
where $e_1, e_2 \in \mathbf{N}$, $e=e_1+e_2$, $C(e_1,e_2)$ and $\Omega(e_1,e_2)$ are numbers defined in a similar fashion as (\ref{COmega}). The map $\ration[\mathbf{N}]^{\Sigfan}_{\omega}\to QH^*(\X)$ defined by sending $\lambda^e$ to the reduced Seidel element $S_e$ is a surjective ring morphism. The kernel is generated by $\mathfrak{P}_{\xi}(\lambda^{y_{1}},...,\lambda^{y_{M}})$'s.
\end{rmk}

\subsection{The Fano Case}\label{sec:Fano}
When the toric orbifold $\X$ is Fano, namely every effective curve has positive Chern number, we have the following lemma.
\begin{lem}\label{FanoGW}
If $B\neq 0$, then $\sum_{i}\< \iota_*f_{i}\>^{\E^{k}}_{0,1,\sigma_{max}+\iota_{*}B}f^{i}=0$.
\end{lem}
The following is an immediate consequence of Lemma \ref{FanoGW}.
\begin{cor}\label{Fanoseidel}
$\tilde{\mathcal{S}}_{k}=X_{k}$. 
\end{cor}

\begin{thm}\label{QHFano}
The orbifold quantum cohomology ring $QH_{orb}^*(\X,\Lambda)$ of a Fano toric symplectic orbifold $(\X,\omega)$ is isomorphic to
$$ \frac{\Lambda[X_1,...,X_M]}{Clos_{\mathfrak{v}_{T}}(\<\sum_{i=1}^{N}\langle e_{\xi},b_{i} \rangle X_{i}^{m_{i}}| \xi=1,...,n\> + \<\prod_{k\in I} X_{k}-q^{C_{I}}T^{\Omega_{I}} \prod_{y_{j}\in \sigma(I)} X_{j}^{c_{j}}|I\in\mathcal{GP} \> + \mathcal{J}(\Sigma))}.$$  
\end{thm}

\begin{proof}[Proof of Lemma \ref{FanoGW}]
If 
$y_{k}\in \text{Gen}(\Sigma)$ such that $\iota_{y_{k}}\neq 0$, then
\begin{eqnarray*}
vdim \overline{\mathcal{M}}_{0,1}(\E^{k},\sigma_{max}+\iota_{*}B,J,(\E^{k}_{(v)})) & & =vdim \overline{\mathcal{M}}_{0,1}(\E^{k},\sigma_{max},J,(\X_{(v)}))+2c_{1}^{\X}(B)\\
& & =2n+2+2(c_{1}^{\X}(B)-\iota_{v}-\iota_{y_{k}}).
\end{eqnarray*}

For $f\in H^{*}(\X_{(v)},\ration)$, $deg \iota_{*}f\le 2n+2-2dim \sigma(v)$. Then by Fanoness,
$$vdim \overline{\mathcal{M}}_{0,1}(\E^{k},\sigma_{max}+\iota_{*}B,J,(\E^{k}_{(v)}))-deg \iota_{*}f\ge 2(c_{1}^{\X}(B)-\iota_{y_{k}}-\iota_{v})+2dim\sigma(v)>0.$$
Consequently, $\< \iota_*f\>^{\E^{k}}_{0,1,\sigma_{max}+\iota_{*}B}= 0$.

Now if $y_{k}\in \text{Gen}(\Sigma)$ such that $\iota_{y_{k}}= 0$, then $y_{k}$ lies in a ray of the fan $\Sigma$ and $dim \sigma(y_{k})=1$. Therefore
\begin{eqnarray*}
vdim \overline{\mathcal{M}}_{0,1}(\E^{k},\sigma_{max}+\iota_{*}B,J,(\E^{k}_{(v)})) & & =2n+2-2dim\sigma(y_{k})+2(c_{1}^{\X}(B)-\iota_{v})\\
& & =2n+2(c_{1}^{\X}(B)-\iota_{v}).
\end{eqnarray*}
Compute
\begin{eqnarray*}
vdim \overline{\mathcal{M}}_{0,1}(\E^{k},\sigma_{max}+\iota_{*}B,J,(\E^{k}_{(v)}))-deg \iota_{*}f & & \ge 2n+2(c_{1}^{\X}(B)-\iota_{v})- (2n+2-2dim \sigma(v))\\
& & =2(c_{1}^{\X}(B)-\iota_{v})-2+2dim\sigma(v).
\end{eqnarray*}

By Fanoness, $2(c_{1}^{\X}(B)-\iota_{v})-2+2dim\sigma(v)=0$ only if $dim\sigma(v)=0$, i.e. $v=0$. Thus $vdim \overline{\mathcal{M}}_{0,1}(\E^{k},\sigma_{max}+\iota_{*}B,J,(\E^{k}_{(v)}))-deg \iota_{*}f=0$ only if $c_{1}^{\X}(B)=1$ and $deg \iota_*f=2n+2$.

In particular, we have shown $\< \iota_*f\>^{\E^{k}}_{0,1,\sigma_{max}+\iota_{*}B}\neq 0$ for $B\neq 0$ is possible only when the evaluation map lands in the trivial twisted sector. Thus from now on the proof  is similar to the manifold case as in \cite{MT}. We will sketch the idea below.

Because $deg \iota_*f=2n+2$, $f=\kappa PD([\X_{(0)}])$ for some nonzero $\kappa\in \ration$. Then homological interpretation of
$$\< \iota_*f)\>^{\E^{k}}_{0,1,\sigma_{max}+\iota_{*}B} \neq 0$$ 
is 
$$\< \iota_*[pt])\>^{\E^{k}}_{0,1,\sigma_{max}+\iota_{*}B} [\X_{(0)}]\neq 0.$$ 
In particular, its intersection product with a point class is non-zero. So
$$\< \iota_*[pt],\iota^{0}_*[pt]\>^{\E^{k}}_{0,2,\sigma_{max}+\iota_{*}B}
\neq 0,$$
where $\iota^{0}$ is the inclusion of a fiber at the south pole. 
 We represent the first point class by the point $x_{max}\in \mathcal{F}_{max,(0)}$ and the second point class by $x_{min}\in \mathcal{F}_{min,(0)}$.

Let $\phi:S^{1}\times S^{2}\to S^{2}$ be the rotation of $S^{2}$ with respect to the south pole and north pole.
Consider a circle action on $\E^{k}$ given by $\phi|_{D_{+}}\times\tilde{\pmb{\gamma}}$ and $\phi|_{D_{-}}\times Id_{\X}$. This action induces a circle action on the moduli space $\overline{\mathcal{M}}_{0,2}(\E^{k}, \sigma_{max}+\iota_{*}B)$. Then by a version of localization proven in \cite{MT}, $\< \iota_*[x_{max}],\iota^{0}_*[x_{min}]\>^{\E^{k}}_{0,2,\sigma_{max}+\iota_{*}B}\neq 0$ only if there exist $S^{1}$-invariant stable orbifold morphisms which consists of constant section determined by $x_{max}$ and a branch component lying in the south pole. The branch component is an $S^{1}$-invariant J-holomorphic curve in $\X$, representing $B$. It is the orbit of a gradient flow of $H_{k}$ from $x_{min}$ to $x_{max}$. So $B=p(\sigma_{max}-\sigma_{min})/q$ for $p\in \inte$ and $q\in \inte_{>0}$.  Thus  $\omega(B)=p (min H_{k}-max H_{k})/q$. Because $\omega(B)>0$ and  $min H_{k}-max H_{k}<0$, so $p>0$.  On the other side $1=c_{1}^{\X}(B)=p(m_{min}-m_{max})$, where $m_{min}$ and $m_{max}$ is the weight of the circle action at $F_{min}$ and $F_{max}$ respectively.  Therefore $m_{min}=m_{max}+\frac{2}{p}\leq 0$, which contradicts the weight at $\mathcal{F}_{min}$ is always positive.
\end{proof}
Now we look at some examples. As a convention we always put the barycenter of the moment polytope at the original. 
\begin{example}\label{CPab}
Weighted projective line $\com P(a,b)$ when $a$ and $b$ are coprime.
\begin{figure}[htb!]\label{fig_F2}
\centering
\begin{tikzpicture}
 \draw (-1,0) node[left] {a} -- (1,0) node[right] {b};
 \draw[snake=brace, raise snake=3, mirror snake] (-1,0)  -- (1,0);
 \draw (0, 0) node[above] {0};
\draw (0, -0.2) node[below] {$\lambda$};
 \draw[-triangle 90] (6,0) node[below] {0} -- (3,0) node[below] {-a};
 \draw[-triangle 90] (6,0) --  (10,0) node[below] {b};        
 \fill[black, fill opacity=0] (0,0) circle (0.03);      
 \fill[black, fill opacity=0] (-1,0) circle (0.05);        
 \fill[black, fill opacity=0] (1,0) circle (0.05);        
 \fill[black, fill opacity=0] (3,0) circle (0.05);        
 \fill[black, fill opacity=0] (4,0) circle (0.05);  
  \fill[black, fill opacity=0] (5,0) circle (0.05);        
 \fill[black, fill opacity=0] (6,0) circle (0.05);
 \fill[black, fill opacity=0] (7,0) circle (0.05);        
 \fill[black, fill opacity=0] (8,0) circle (0.05);  
  \fill[black, fill opacity=0] (9,0) circle (0.05);        
 \fill[black, fill opacity=0] (10,0) circle (0.05);
\end{tikzpicture} 
\caption{Labelled Moment Polytope and Stacky Fan of $\com P(a,b)$.}
\end{figure}
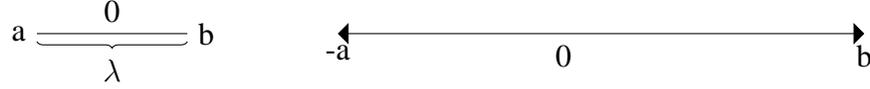
$$ QH^{*}_{orb}(\com P(a,b),\Lambda)\cong \frac{\Lambda[X_1,X_2]}{Clos_{\mathfrak{v}_{T}}(\<-aX_{1}^{a}+bX_{2}^{b}\> + \<X_{1}X_{2}-q^{\frac{1}{a}+\frac{1}{b}}T^{\lambda}\> )}.$$  
\end{example}

\begin{example}\label{F2}
Weighted projective space $\com P(1,1,2)$.
\begin{figure}[htb!]\label{fig_F2}
\centering
\begin{tikzpicture}
 \draw (-6,0) node[below left] {$(-\frac{2}{3}\lambda,-\frac{1}{3}\lambda)$} --  (-4,0) node[below right] {$(\frac{4}{3}\lambda,-\frac{1}{3}\lambda)$} --  (-6,1) node[above left] {$(-\frac{2}{3}\lambda,\frac{2}{3}\lambda)$} -- cycle;
 \draw[snake=brace, raise snake=10, mirror snake] (-6,0)  -- (-4,0);
 \draw[snake=brace, raise snake=10] (-6,0)  -- (-6,1);
 \draw[snake=brace, raise snake=10, mirror snake] (-4,0)  -- (-6,1);
 \draw (-5,-0.5) node[below] {$2\lambda$};
 \draw (-6.5,0.5) node[left] {$\lambda$};
 \draw (-5,1) node[above right] {$\sqrt{5}\lambda$};
 \draw[-triangle 90] (0,0) node[below right] {(0,0)} -- (-1,0) node[below] {$y_{1}$};
  \draw[-triangle 90] (0,0) -- (0,-1) node[below right] {$y_{2}$};
 \draw[-triangle 90] (0,0) -- (1,2) node[right] {$y_{3}$};
 \draw (0,1) node[left] {$y_{4}$};
   \fill[black, fill opacity=0] (-5.33,0.33) circle (0.03);        
 \fill[black, fill opacity=0] (0,0) circle (0.05);        
 \fill[black, fill opacity=0] (0,1) circle (0.05);     
 \fill[black, fill opacity=0] (1,0) circle (0.05);        
 \fill[black, fill opacity=0] (0,2) circle (0.05);        
 \fill[black, fill opacity=0] (2,0) circle (0.05);        
 \fill[black, fill opacity=0] (1,1) circle (0.05);        
 \fill[black, fill opacity=0] (2,1) circle (0.05);   
  \fill[black, fill opacity=0] (1,2) circle (0.05);        
 \fill[black, fill opacity=0] (2,2) circle (0.05);   
  \fill[black, fill opacity=0] (0,-1) circle (0.05);        
 \fill[black, fill opacity=0] (1,-1) circle (0.05);        
 \fill[black, fill opacity=0] (2,-1) circle (0.05);       
  \fill[black, fill opacity=0] (-1,-1) circle (0.05);         
 \fill[black, fill opacity=0] (-1,0) circle (0.05);        
 \fill[black, fill opacity=0] (-1,1) circle (0.05);        
 \fill[black, fill opacity=0] (-1,2) circle (0.05);        
\end{tikzpicture} 
\caption{Moment Polytope and Fan of $F_{2}$.}
\end{figure}
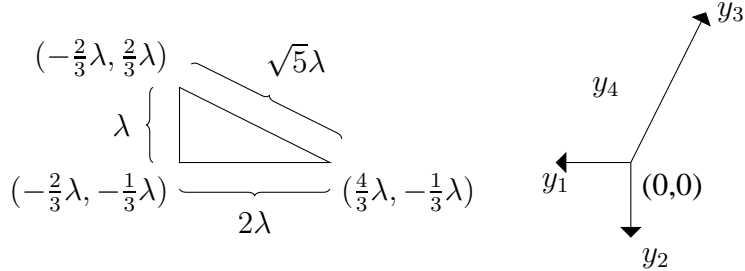

Then the orbifold quantum cohomology ring $QH^{*}_{orb}(\com P(1,1,2),\Lambda)$ is isomorphic to
$$ \frac{\Lambda[X_1,X_2,X_{3},X_{4}]}{Clos_{\mathfrak{v}_{T}}(\<-X_{1}+X_{3},-X_{2}+2X_{3}\> + \<X_{1}X_{2}X_{3}-q^{C}T^{\Omega}X_{4}\>+\<X_{1}X_{3}-X_{4}^{2}\> )},$$  
where 
$$C=1+1+1-\frac{1}{2}-\frac{1}{2}=2,$$
$$\Omega=\frac{2}{3}\lambda+\frac{1}{3}\lambda+\frac{2}{3}\lambda-\frac{1}{2}\cdot\frac{2}{3}\lambda-\frac{1}{2}\cdot\frac{2}{3}\lambda=\lambda.$$
Note that the generalized primitive collections in this example are $\{y_{1},y_{2},y_{3}\}$ and $\{y_{2},y_{4}\}$. But the quantum Stanley-Reisner relation associated to $\{y_{2},y_{4}\}$ is already contained in $$ \<X_{1}X_{2}X_{3}-q^{C}T^{\Omega}X_{4}\>+\<X_{1}X_{3}-X_{4}^{2}\>.$$

\end{example}

\end{document}